\documentclass{amsart}

\usepackage{verbatim}
\usepackage{amssymb}
\usepackage{amsbsy}
\usepackage{amscd}
\usepackage{amsmath}
\usepackage{amsthm}
\usepackage[mathscr]{eucal}

\usepackage{amsfonts}
\usepackage{graphics}
\usepackage{epsfig}
\usepackage[all]{xy}
\usepackage{latexsym}
\usepackage{graphicx}

\newtheorem{thm}{Theorem}[section]
\newtheorem{Def}{Definition}
\numberwithin{Def}{section}

\newtheorem{lem}[thm]{Lemma}
\newtheorem{prop}[thm]{Proposition}
\theoremstyle{definition}
\newtheorem{rem}[thm]{Remark}

\numberwithin{equation}{section}
\numberwithin{figure}{section}

\def\Hom{{\text{\rm{Hom}}}}
\def\End{{\text{\rm{End}}}}

\def\tr{{\text{\rm{tr}}}}
\def\rchi{{\hbox{\raise1.5pt\hbox{$\chi$}}}}
\def\rgam{{\hbox{\raise1.5pt\hbox{$\gamma$}}}}

\def\isom{\cong}
\def\tensor{\otimes}
\def\dsum{\oplus}
\def\Ker{{\text{\rm{Ker}}}}
\def\Coker{{\text{\rm{Coker}}}}

\def\index{{\text{\rm{index}}}}

\def\reg{{\text{\rm{reg}}}}
\def\res{{\text{\rm{res}}}}

\def\lam{\lambda}

\def\rank{{\text{\rm{rank}}\,}}
\def\Sym{{\text{\rm{Sym}}}}
\def\Spec{{\text{\rm{Spec}}}}
\def\Jac{{\text{\rm{Jac}}}}
\def\Pic{{\text{\rm{Pic}}}}
\def\Prym{{\text{\rm{Prym}}}}
\def\Nm{{\text{\rm{Nm}}}}
\def\Res{{\text{\rm{Res}}}}

\title[Hitchin systems, spectral curves, and KP equations]{Hitchin integrable systems, deformations of spectral curves, and KP-type equations}

\author[Andrew Hodge]{Andrew R. Hodge$^1$}  
\address{
(Address after September 2008)
National Security Agency\\
Fort George G.~Meade, MD 20755-6000}
\thanks{$^1$Research supported by NSF grants
DMS-0135345 (VIGRE) and DMS-0406077
while at UC Davis.}

\author[Motohico Mulase]{Motohico Mulase$^2$}  
\address{
Department of Mathematics\\
University of California\\
Davis, CA 95616--8633}
\email{mulase@math.ucdavis.edu}
\thanks{$^2$Research supported by NSF grant DMS-0406077 and UC Davis.}


\begin{document}

\begin{abstract}
An effective family of spectral curves appearing in
Hitchin fibrations is determined. Using this family
the moduli spaces of stable Higgs bundles on an
algebraic curve are embedded into
the Sato Grassmannian. We show that
the Hitchin integrable system,
the natural algebraically completely integrable 
Hamiltonian system defined
on the Higgs moduli space, coincides with the 
KP equations. It is shown that the Serre duality
on these moduli spaces corresponds to the formal adjoint
of pseudo-differential operators acting on the 
Grassmannian. From this fact we then identify 
the Hitchin integrable system on the moduli 
space of $Sp_{2m}$-Higgs bundles in terms of a reduction 
of the KP
equations. We also show that the dual Abelian fibration
(the SYZ mirror dual) to the $Sp_{2m}$-Higgs
moduli space
is constructed by  taking the symplectic quotient 
of a Lie \emph{algebra} action on the moduli space of $GL$-Higgs
bundles.
\end{abstract}

\maketitle

\allowdisplaybreaks


\section{Introduction}
\label{sect:intro}

The purpose of this paper is to determine the 
relation between the KP-type equations defined
on the Sato Grassmannians and the Hitchin
integrable systems defined on the moduli spaces
of stable Higgs bundles.
The results established are the following:
\begin{enumerate}
\item 
We  determine the \emph{effective}
family of spectral curves appearing in the
Hitchin fibration of the moduli spaces of 
stable Higgs bundles. 
\item
We embed the effective family of Jacobian varieties
of the spectral curves into the Sato Grassmannian and
show that the KP flows are tangent to each fiber
of the Hitchin fibration. 
\item
The moduli space of Higgs bundles of rank $n$ and
degree $n(g-1)$ on an algebraic curve of genus $g\ge 2$
is embedded into the \emph{relative}
Grassmannian of \cite{AMP, DM, Q}. Using this embedding
we show that the Hitchin integrable system is exactly the
restriction of the KP equations on the Grassmannian to the 
image of this embedding.
\item
It is shown that the Krichever construction transforms 
the Serre duality of the geometric data consisting of
algebraic curves and vector bundles on them to
the formal adjoint of pseudo-differential operators 
acting on the Grassmannian.
By identifying the fixed-point-set of the Serre duality and
the formal adjoint operation we determine the KP-type
equations that are equivalent to the Hitchin integrable 
system defined on the moduli space of $Sp_{2m}$-Higgs 
bundles.
\item
There are two ways to \emph{reduce}
an algebraically completely integrable Hamiltonian system:
one by restriction and the other by taking a quotient of a Lie 
algebra action that is similar to the symplectic quotient. 
When applied to the moduli spaces of Higgs bundles,
these constructions yield
SYZ-mirror pairs. We interpret
the $SL$-$PGL$ and $Sp_{2m}$-$SO_{2m+1}$ dualities
in this way.
\end{enumerate}

Let $\mathcal{G}_\mathbb{C}$ be the category of 
complex Lie groups, and $\mathcal{CY}$ the category
of Calabi-Yau spaces. 
For  a compact oriented surface $\Sigma$ 
of genus $g\ge 2$, 
the functor
\begin{equation*}
\Hom(\hat{\pi}_1(\Sigma), \;\cdot\;)/\!\!/\;\cdot\;:
\mathcal{G}_\mathbb{C}\longrightarrow
\mathcal{CY}
\end{equation*}
assigns to each complex Lie group $G$ its 
character variety
$$\Hom\big(\hat{\pi}_1(\Sigma), G\big)/\!\!/G, 
$$
where $\hat{\pi}_1(\Sigma)$ is the central 
extension of the fundamental group of $\Sigma$. The quotient 
by conjugation is the geometric invariant theory
quotient of Mumford \cite{MFK}. An amazing discovery
of Hausel and Thaddeus \cite{HT}, and its generalizations 
by \cite{DP, KW} and others, is that 
the character variety functor transforms the Langlands duality
in $\mathcal{G}_\mathbb{C}$ to the
mirror symmetry of Calabi-Yau spaces in the sense of 
Strominger-Yau-Zaslow \cite{SYZ}:
\begin{equation*}
\begin{CD}
\mathcal{G}_\mathbb{C} @>{
\Hom(\hat{\pi}_1(\Sigma), \;\cdot\;)/\!\!/\;\cdot\;}>>
\mathcal{CY}\\
@V{\text{Langlands Dual}}VV @VV{\text{Mirror Dual}}V\\
\mathcal{G}_\mathbb{C} @>>{
\Hom(\hat{\pi}_1(\Sigma), \;\cdot\;)/\!\!/\;\cdot\;}>
\mathcal{CY}
\end{CD}
\end{equation*}
The character
variety 
$\Hom\big(\hat{\pi}_1(\Sigma), G\big)/\!\!/G$
has many distinct complex structures
\cite{HT,KW}. To understand the SYZ mirror symmetry among 
the character varieties,
it is most convenient 
to realize them as \emph{Hitchin integrable systems}.
In his seminal papers \cite{H1,H2}, Hitchin identifies the character
variety 
with  the 
moduli space of stable $G$-\emph{Higgs bundles}, 
which has the structure
of an 
\emph{algebraically completely integrable Hamiltonian system}.

An algebraically completely integrable Hamiltonian system 
\cite{DM, V} is 
a holomorphic symplectic manifold $(X,\omega)$ 
of dimension $2N$ together with
a holomorphic map $H:X\rightarrow 
\mathfrak{g}^*$ such that
\begin{enumerate}
\item a general fiber $H^{-1}(s)$, $s\in \mathfrak{g}^*$, is 
an Abelian variety of dimension $N$, 
\item $\mathfrak{g}^*$ is the dual Lie algebra of a general
fiber $H^{-1}(s)$ considered as a Lie group, and
\item the coordinate components of the map 
$H$ are Poisson commutative
with respect to the symplectic structure $\omega$.
\end{enumerate}

The notion corresponding
to an algebraically completely integrable Hamiltonian system
 in \emph{real} symplectic geometry
 is the cotangent bundle of 
a torus. The procedure of symplectic quotient
is to remove the effect of this cotangent bundle
from a given symplectic manifold. 
In the holomorphic context, it is often useful to 
take the quotient by a \emph{family} of groups
that have the same Lie algebra.
Suppose we have a Lie algebra direct sum decomposition
$\mathfrak{g}=\mathfrak{g}_1\dsum \mathfrak{g}_2$.
If  $\mathfrak{g}_1$-action on $X$ is integrable to a group
$G_{1,s}$-action in each fiber $H^{-1}(s)$
for $s\in \mathfrak{g}_2 ^*$, then we can define 
a \emph{quotient}
$X/\!\!/\mathfrak{g}_1$ as the family of
quotients $H^{-1}(s)/G_{1,s}$ over $\mathfrak{g}_2 ^*$.
We can also construct a \emph{reduction} of
$(X,\omega,H)$ by restricting the fibration to $\mathfrak{g}_2 ^*$
and considering the family of $\mathfrak{g}_2$-orbits in $H^{-1}(s)$, 
if the $\mathfrak{g}_2$-action is integrated to a group action over
$\mathfrak{g}_2 ^*$.
When applied to the moduli space of Higgs bundles, these two
constructions yield Abelian fibrations that are dual to one another,
producing an SYZ mirror pair.
We examine these constructions for the $SL$-$PGL$ and
$Sp_{2m}$-$SO_{2m+1}$ dualities.

From the results established in \cite{DM, LM2, M1984, M1990},
we know that linear integrable evolution equations on the
Jacobians or Prym varieties are realized as the restriction of
KP-type equations defined on the Sato Grassmannians
through a generalization of Krichever construction.
Since the Hitchin integrable systems are defined on a
\emph{family} of Jacobian varieties  or Prym varieties, we need to 
embed the whole family into the Sato Grassmannian 
to compare the Hitchin systems and the KP equations.
To deal with families, we use two different approaches
in this article.
One approach is to utilize the theory of Sato Grassmannians
defined over an arbitrary scheme developed in 
 \cite{AMP, DM, PM, Q}. In this way we can directly compare
the integrable Hamiltonian systems on the Higgs moduli spaces
and the KP equations. 
The other approach is to examine the deformations of spectral
curves that appear in the Hitchin Hamiltonian systems. 
Once we identify the effective family of spectral curves,
we can embed the whole family into a single Sato
Grassmannian over $\mathbb{C}$, using the method
developed in \cite{LM2}.

The second approach has an unexpected application: we can
identify the effect of Serre duality operation on
the algebro-geometric data in terms of the language of Grassmannians.
Note that Sato Grassmannians are constructed from 
pseudo-differential operators \cite{SS, SW}. We will show,
using Abel's theorem,
that the Serre duality
is simply the formal adjoint operation on the pseudo-differential
operators. 
Since the $Sp$-Hitchin system is the
fixed-point-set  of the Serre duality on the $GL$-Hitchin system,
we can determine the integrable equations corresponding to 
the $Sp$ case
as a reduction of the KP equations on the
fixed-point-set  of the formal adjoint action on pseudo-differential
operators. Since our formal adjoint is slightly different from
what has been studied in the literature \cite{JM, S2, T},
the equations coming up for the symplectic 
groups are \emph{not} BKP or CKP equations. Let
$$
P = \sum a_i(x) \left(\frac{\partial}{\partial x}\right)^i
$$
be a formal pseudo-differential operator, where 
$a_i(x)$ is a matrix valued functions. We define the \emph{formal
adjoint} by
$$
P^* = \sum \left(\frac{\partial}{\partial x}\right)^i \cdot a_i(-x)^t .
$$
The reduction of the KP equations that corresponds to the 
$Sp$-Hitchin system is the $2m$-component KP equations
that preserve the algebraic condition
\begin{equation}
\label{eq:SpLaxcondition}
\mathbf{L}^* = 
\begin{bmatrix}
&I_m\\
I_m&\\
\end{bmatrix}
\cdot
\mathbf{L}
\cdot 
\begin{bmatrix}
&I_m\\
I_m&\\
\end{bmatrix}
\end{equation}
for a $2m\times 2m$ matrix Lax operator $\mathbf{L}$ with 
the leading
term $I_{2m}\cdot \partial/\partial x$.
Several authors  have proposed integrable systems with
$Sp_{2m}$-symmetry (cf. \cite{TT}). It would be 
interesting to study  our reduction 
(\ref{eq:SpLaxcondition}) from the point of view
of integrable systems and to investigate the relation with  the
other $Sp$ integrable 
systems.

The fundamental literature of the algebro-geometric
study of the Hitchin 
integrable systems and related topics
is the book \cite{DM} by Donagi and Markman. 
Our present paper employs a slightly different perspective,
that leads to the discovery of the KP-type equations
corresponding to the $Sp$ Hitchin system.

The relation between
the Hitchin integrable systems and the KP equations was
also studied in \cite{LM1}. The treatment there was
limited to the study of the Hitchin system on a single fiber.
The present article extends the result therein. 

We also note that many topics of this paper have been 
studied by the Salamanca school of algebraic geometers
from yet another
 point of view \cite{AMP, GGMPPM, HSMPPM, PM}.

The paper is organized as follows.
The first section is devoted to reviewing 
the Hitchin integrable systems of \cite{BNR, H2}.
We then determine an effective family of spectral curves
in Section~\ref{sect:deformation}. 
Section~\ref{sect:quotient} is devoted to giving two constructions
of reduced integrable systems from a Hitchin system:
one is a straightforward specialization, and the other is
a kind of symplectic reduction by a Lie subalgebra.
These two constructions
give rise to an Abelian fibration and its dual Abelian fibration.
We show that the $GL$-Hitchin integrable system is equivalent
to the KP equations in Section~\ref{sect:KP}. 
The identification of the Serre duality in terms of 
Grassmannians and pseudo-differential operators is carried out
in Section~\ref{sect:Serre}. Finally we determine the 
KP-type equations for the $Sp$ Hitchin system.

\section{Hitchin integrable systems}
\label{sect:Hitchin}

In this section we review the algebraically completely
integrable Hamiltonian systems defined on the moduli
spaces of Higgs bundles, following \cite{BNR, DM, H1, H2}.

Throughout the paper we denote by $C$
 a non-singular algebraic curve  of genus
$g\ge 2$. The moduli space of semistable algebraic 
vector bundles on $C$ of rank $n$ and degree $d$ is denoted
by $\mathcal{U}_C(n,d)$. 
When $n$ and $d$ are relatively 
prime,  a semistable bundle is automatically stable, and the moduli
space is a non-singular projective algebraic variety of 
dimension $n^2(g-1) +1$. We denote by
\begin{equation}
\label{eq:UCn}
\mathcal{U}_C(n) = \coprod_{d\in\mathbb{Z}} \mathcal{U}_C(n,d)
\end{equation}
the space of all stable vector bundles.
A \emph{Higgs bundle} is a pair $(E,\phi)$ consisting of an
algebraic vector bundle $E$ on $C$ and a global section
\begin{equation}
\label{eq:phi}
\phi\in H^0(C,\End(E)\tensor K_C)
\end{equation}
of the endomorphism sheaf of $E$ twisted by the canonical sheaf
$K_C$ of $C$. A Higgs bundle  is  \emph{stable} if 
$
\frac{\deg F}{\rank F}<\frac{\deg E}{\rank E}
$
for every $\phi$-invariant proper holomorphic vector subbundle $F$.
An endomorphism of a Higgs bundle $(E,\phi)$ is 
 an endomorphism $\psi$ of $E$ that commutes with $\phi$:
\begin{equation*}
\begin{CD}
E @>\psi>> E\\
@V{\phi}VV @VV{\phi}V\\
E\tensor K_C@>>{\psi\tensor 1}> E\tensor K_C 
\end{CD}
\end{equation*}
It is known that 
$H^0(C,\End(E,\phi)) = \mathbb{C}$
for a stable Higgs bundle, and one can define the moduli space
of stable objects.
We denote by $\mathcal{H}_C(n,d)$ the moduli space of 
stable Higgs bundles of rank $n$ and degree $d$ on
$C$, and
\begin{equation}
\label{eq:HCn}
\mathcal{H}_C(n) = \coprod_{d\in \mathbb{Z}}\mathcal{H}_C(n,d).
\end{equation}
Note that the \emph{Serre duality} induces an involution on 
$\mathcal{H}_C(n)$ defined by
\begin{equation}
\label{eq:serreduality}
\mathcal{H}_C(n,d)\owns (E,\phi) \longmapsto
(E^*\tensor K_C,-\phi^*)\in \mathcal{H}_C(n,-d+2n(g-1)) .
\end{equation}
 The dual of the
Higgs field
$\phi: E\rightarrow E\tensor K_C$ is a homomorphism
$\phi^*:E^*\tensor K_C ^{-1}\rightarrow E^*$. We use the same notation
for  the homomorphism 
$E^*\tensor K_C\rightarrow E^*\tensor K_C ^{\tensor 2}$
induced by $\phi^*$.

If $E$ is stable, then $(E,\phi)$ is stable for any
$\phi$ of (\ref{eq:phi}). And if $\phi=0$, then the stability of $(E,\phi)$
simply means $E$ is stable. Therefore, 
the Higgs moduli space contains the total space of the
holomorphic
cotangent bundle
\begin{equation}
\label{eq:cotangentinH}
T^* \mathcal{U}_C(n,d)\subset \mathcal{H}_C(n,d),
\end{equation}
since
$$
H^0(C,\End(E)\tensor K_C)
\isom H^1(C,\End(E))^* \isom T^* _E\, \mathcal{U}_C(n,d).
$$
   Note that
 $p^* \Lambda^1 (\mathcal{U}_C(n,d))
\subset \Lambda^1(T^*\mathcal{U}_C(n,d))$
has a tautological section 
\begin{equation}
\label{eq:eta}
\eta\in H^0(T^*\mathcal{U}_C(n,d), 
p^* \Lambda^1 (\mathcal{U}_C(n,d))),
\end{equation}
where $
p:T^*\mathcal{U}_C(n,d)\rightarrow 
\mathcal{U}_C(n,d)
$ is the projection, and 
 $\Lambda^r$ denotes  the sheaf of
holomorphic $r$-forms. 
The differential $\omega = - d\eta$ of the 
tautological section defines the canonical holomorphic symplectic
form  on $T^*\mathcal{U}_C(n,d)$. 
The restriction of $\omega$ on $\mathcal{U}_C(n,d)$,
which is the $0$-section of the cotangent bundle, is identically
$0$. Therefore the $0$-section is a Lagrangian submanifold
of this holomorphic symplectic space.

Let us denote by
\begin{equation}
\label{eq:baseGL}
V^* _{GL} = V^* _{GL_n(\mathbb{C})} = 
\bigoplus_{i=1} ^ n 
H^0(C,K_C ^{\tensor i}).
\end{equation}
As a vector space $V^* _{GL}$ has the same 
dimension of
$
H^0(C,\End(E)\tensor K_C) 
$.
The Higgs field $\phi: E \rightarrow E\tensor K_C$
induces a homomorphism of the $i$-th anti-symmetric 
tensor product spaces
$$
\wedge^i (\phi): \wedge^i (E) \longrightarrow 
\wedge^i (E\tensor K_C) = \wedge^i(E) \tensor K_C ^{\tensor i},
$$
or equivalently $\wedge^i (\phi)\in
H^0(C,\End(\wedge^i(E))\tensor K_C ^{\tensor i})$.
Taking its natural trace, we obtain
$$
\tr \wedge^i(\phi) \in H^0(C,K_C ^{\tensor i}).
$$
This is exactly the $i$-th characteristic coefficient of
the twisted endomorphism $\phi$:
\begin{equation}
\label{eq:characteristic}
\det (x - \phi) = x^n + \sum_{i=1} ^n (-1)^i \tr \wedge^i(\phi) \cdot
x^{n-i}.
\end{equation}
By assigning the coefficients of (\ref{eq:characteristic}),
 Hitchin \cite{H2} defines a
holomorphic map, now known as the \emph{Hitchin fibration}
or \emph{Hitchin map},
\begin{equation}
\label{eq:Hitchinmap}
H: \mathcal{H}_C (n,d)\owns (E,\phi)\longmapsto
\det(x-\phi)\in \bigoplus_{i=1} ^ n 
H^0(C,K_C ^{\tensor i}) = V^* _{GL}.
\end{equation}
The map $H$ to the vector space $V^* _{GL}$ is a collection
of $N= n^2(g-1)+1$ globally defined holomorphic
functions on $\mathcal{H}_C (n,d)$. The  $0$-fiber of the Hitchin fibration is the moduli space $\mathcal{U}_C(n,d)$.

To determine the generic fiber of $H$, 
the notion of \emph{spectral curves} is introduced in 
\cite{H2}. 
 The total space of the canonical sheaf $K_C = \Lambda^1(C)$
on $C$ is the cotangent bundle $T^*C$. Let 
\begin{equation}
\label{eq:pi}
\pi:T^*C \longrightarrow C
\end{equation}
be the projection, and 
$$
\tau\in H^0(T^*C, \pi^*K_C)\subset H^0(T^*C,\Lambda^1(T^*C))
$$
be the tautological section of $\pi^*K_C$ on $T^*C$. Here again
$\omega_{T^*C} = -d\tau$ is the holomorphic symplectic
form on $T^*C$. The tautological section $\tau$ satisfies
that $\sigma^*\tau = \sigma$
for every section $\sigma \in H^0(C,K_C)$ viewed as a holomorphic
map $\sigma:C\rightarrow T^*C$. The characteristic coefficients 
(\ref{eq:characteristic})
of $\phi$ 
 give a section
\begin{equation}
\label{eq:chareq}
s=\det(\tau-\phi) = \tau ^{\tensor n} + 
\sum_{i=1} ^n (-1)^{i}
\tr \wedge^i(\phi)\cdot \tau^{\tensor n-1}
\in H^0(T^*C,\pi^* K_C ^{\tensor n}).
\end{equation}
We define the spectral curve $C_s$ associated with a Higgs pair
$(E,\phi)$ as the divisor of $0$-points of the section $s=\det(\tau-\phi)$ of
the line bundle $\pi^*K_C ^{\tensor n}$:
\begin{equation}
\label{eq:spectralcurve}
C_s = (s)_0 \subset T^*C .
\end{equation}
 The projection $\pi$ of (\ref{eq:pi}) defines
a ramified covering map 
$
\pi:C_s\rightarrow C
$ of degree $n$.

There is yet
another construction  of the spectral curve $C_s$. 
Since the section $s=\det(\tau-\phi)$ is
determined by the characteristic coefficients of $\phi$, by
abuse of notation we identify $s$ as an element of $V^* _{GL}$:
\begin{multline*}
s = (s_1,s_2,\dots, s_n) = (-\tr\, \phi,
\tr \wedge^2 (\phi), \dots, (-1)^n \tr \wedge^n(\phi))\\
\in \bigoplus_{i=1} ^n H^0(C,K_C ^{\tensor i}).
\end{multline*}
It defines an $\mathcal{O}_C$-module 
$(s_1+s_2+\cdots +s_n)\cdot K_C ^{\tensor -n}$. Let 
$\mathcal{I}_s$ denote  the ideal generated by this module
in the symmetric tensor algebra $\Sym(K_C ^{-1})$. Since
$K_C ^{-1}$ is the sheaf of linear functions on $T^*C$, the scheme
associated to this tensor algebra is 
$\Spec \big(\Sym (K_C ^{-1}) \big)= T^*C$. The 
spectral curve as the divisor 
of $0$-points of $s$ is then defined by
\begin{equation}
\label{eq:spectralcurve2}
C_s = \Spec\left(\frac{\Sym(K_C ^{-1})}{\mathcal{I}_s}\right)
\subset \Spec \big(\Sym (K_C ^{-1}) \big)= T^*C .
\end{equation}
We denote by $U_\reg$ the set
 consisting of points $s$ for which $C_s$ is  
 non-singular. It
is an open dense subset
of $V^* _{GL}$ \cite{BNR}. We note that since $\deg(K_C) = 2g-2>0$, 
every divisor of $T^*C$ intersects with the $0$-section $C$.
Therefore, if $C_s$ is non-singular, then it has only one component,
and is therefore irreducible. 
The genus of $C_s$ is
$g(C_s) = n^2(g-1) + 1$, which follows from an isomorphism
$$
\pi_* \mathcal{O}_{C_s} = \Sym (K_C ^{-1})/\mathcal{I}_s
\isom \bigoplus_{i=0} ^{n-1} K_C^{\tensor -i}
$$ 
as an $\mathcal{O}_C$-module.

The Higgs field $\phi\in H^0(C,\End(E)\tensor K_C)$ gives
a homomorphism 
$$
\varphi: K_C ^{-1} \longrightarrow \End(E),
$$
which induces an algebra homomorphism, still denoted by the
same letter,
$$
\varphi : \Sym (K_C ^{-1}) \longrightarrow \End(E).
$$
Since $s\in V^* _{GL}$ is the characteristic coefficients of $\varphi$,
by the Cayley-Hamilton theorem,
 the homomorphism $\varphi$ factors through
$$
\Sym (K_C ^{-1}) \longrightarrow \Sym (K_C ^{-1})/\mathcal{I}_s
\longrightarrow
\End(E).
$$
Hence $E$ is a module over 
$\Sym (K_C ^{-1})/\mathcal{I}_s$ of rank $1$. The rank is 
$1$ because
the ranks of $E$ and
 $\Sym (K_C ^{-1})/\mathcal{I}_s$
 are the same as $\mathcal{O}_C$-modules. 
In this way a Higgs pair $(E,\phi)$ gives rise to a line bundle
$\mathcal{L}_E$ on the spectral curve $C_s$, if it is nonsingular.
Since $\mathcal{L}_E$ being an $\mathcal{O}_{C_s}$-module
is equivalent to $E$ being a 
$\Sym (K_C ^{-1})/\mathcal{I}_s$-module, we recover $E$
from $\mathcal{L}_E$ simply by
$E = \pi_* \mathcal{L}_E$, which has rank $n$ because $\pi$ is a covering
of degree $n$. From the equality 
$\rchi(C,E) = \rchi(C_s,\mathcal{L}_E)$ and Riemann-Roch, we find
that $\deg \mathcal{L}_E = d+ n(n-1) (g-1)$.
To summarize, the above construction defines an inclusion
map
$$
H^{-1}(s) \subset  \Pic^{d+n(n-1)(g-1)}(C_s) \isom \Jac(C_s),
$$
if $C_s$ is  non-singular.

Conversely,
suppose we have a line bundle $\mathcal{L}$ of degree $d+n(n-1)(g-1)$
on a  non-singular spectral curve 
$C_s$. Then $\pi_*\mathcal{L}$ is a module over
$\pi_*\mathcal{O}_{C_s}
=\Sym (K_C ^{-1})/\mathcal{I}_s$, which
defines a homomorphism $K_C^{-1}\rightarrow
\End(\pi_*\mathcal{L})$, or equivalently,
$\psi:\pi_*\mathcal{L}\rightarrow
\pi_*\mathcal{L}\tensor K_C$. It is easy to 
see that the Higgs pair $(\pi_*\mathcal{L},\psi)$ is stable. Suppose
we had a $\psi$-invariant subbundle $F\subset \pi_*\mathcal{L}$
of rank $k<n$. Since $(F,\psi|_F)$ is a Higgs pair,
it gives rise to a spectral curve $C_{s'}$ that is a degree $k$ covering
of $C$. Note that the characteristic polynomial
$s' = \det(x-\psi)$ is a factor of the full characteristic polynomial
$s=\det(x-\phi)$. 
Therefore, $C_{s'}$ is a component of $C_s$, 
which contradicts to our assumption that
$C_s$ is irreducible.
Therefore, $\pi_*\mathcal{L}$ has no $\psi$-invariant proper subbundle. 
Thus we have established that
\begin{equation}
\label{eq:Hitchinfiber}
H^{-1}(s)= \Pic^{d+n(n-1)(g-1)}(C_s) \isom \Jac(C_s), \qquad s\in U_\reg\subset V^* _{GL}.
\end{equation}
The vector bundle $\pi_*\mathcal{L}$ itself is not necessarily 
stable. 
It is proved in \cite{BNR} that the locus of $\mathcal{L}$ in
$\Pic^{d+n(n-1)(g-1)}(C_s)$ that gives unstable
$\pi_*\mathcal{L}$ has codimension two or more.

Recall that  the tautological section
$$\eta\in H^0(T^*\mathcal{U}_C(n,d),p^*\Lambda^1(
\mathcal{U}_C(n,d)))$$ is a  holomorphic
$1$-form on $T^*\mathcal{U}_C(n,d)
\subset \mathcal{H}_C(n,d)$. Its restriction to the fiber $H^{-1}(s)$
of  $s\in U_\reg$ for which $C_s$ is non-singular
extends to a holomorphic $1$-form on the whole 
fiber $H^{-1}(s)\isom \Jac(C_s)$ since
$\eta$ is undefined only on a codimension $2$ subset.
Hence $\eta$ extends as a holomorphic $1$-form
on $H^{-1}(U_\reg)$.
Thus $\eta$ is well defined on 
$T^*\mathcal{U}_C(n,d)\cup H^{-1}(U_\reg)$. The complement of
this open subset in $\mathcal{H}_C(n,d)$ consists of such Higgs pairs
$(E,\phi)$ that $E$ is unstable \emph{and} $C_s$ is 
singular. Since the stability of $E$ and the non-singular condition 
for $C_s$ are
both open conditions, this complement has codimension at least 
two. Consequently, both the tautological section $\eta$ and the 
holomorphic symplectic form $\omega =- d\eta$ extend holomorphically
to the whole Higgs moduli space $\mathcal{H}_C(n,d)$.

We note that there are no holomorphic $1$-forms other than  constants
on a Jacobian variety since its cotangent bundle is trivial.  It implies  that
$$
\omega|_{H^{-1}(s)} = -d(\eta|_{H^{-1}(s)}) \equiv 0
$$
for $s\in U_\reg$. Therefore, a generic fiber of the Hitchin fibration is 
a Lagrangian subvariety of the holomorphic
symplectic variety $(\mathcal{H}_C(n,d),\omega)$.
The \emph{Poisson structure}
on $H^0(\mathcal{H}_C(n,d),\mathcal{O}_{\mathcal{H}_C(n,d)})$ is defined by
$$
\{f,g\} = \omega(X_f,X_g),\qquad f,g\in 
H^0(\mathcal{H}_C(n,d),\mathcal{O}_{\mathcal{H}_C(n,d)}),
$$
where $X_f$ denotes the Hamiltonian vector field defined by
 the relation $df= \omega(X_f,\cdot)$. 
Since $\omega$ vanishes on a generic  fiber of $H$, the holomorphic functions on 
$\mathcal{H}_C(n,d)$ coming from the coordinates
of the Hitchin fibration
 are \emph{Poisson commutative} with respect to the
holomorphic symplectic structure $\omega$. 
Therefore, $(\mathcal{H}_C(n,d),\omega, H)$ is an algebraically
completely integrable Hamiltonian system \cite{DM, V},
called the \emph{Hitchin integrable system}.

\begin{thm}[Hitchin \cite{H2}]
The Hitchin fibration
$$
H: \mathcal{H}_C(n,d)\longrightarrow V^* _{GL}
$$
is a generically Lagrangian Jacobian fibration
that defines an algebraically completely integrable Hamiltonian
system $(\mathcal{H}_C(n,d),\omega, H)$.  A
generic fiber $H^{-1}(s)$
is a Lagrangian with respect to the holomorphic symplectic 
structure $\omega$ and is isomorphic to the Jacobian variety
of a  spectral curve $C_s$.
\end{thm}

The Hitchin map $H$ restricted
to $T_E  ^* \, \mathcal{U}_C(n,d)$ for a stable $E$
is a polynomial map
$$
T_E ^*\,  \mathcal{U}_C(n,d)= H^0(C,\End(E)\tensor K_C)\owns \phi
\longmapsto \det(x-\phi)\in
\bigoplus_{i=1} ^n H^0(C,K_C ^{\tensor i})
$$
consisting of $g$ linear components, $3g-3$ quadratic
components, $5g-5$ cubic components, etc., and $(2n-1)(g-1)$
components of degree $n$. Thus the inverse image
$$
H^{-1}(s)\cap T_E ^* \,\mathcal{U}_C(n,d)
$$
for a generic $(E,\phi)$ consists of 
\begin{equation}
\label{eq:delta}
\delta = \prod_{i=1} ^n i^{(2i-1)(g-1)}
\end{equation}
points  \cite{BNR}. Consequently, the map
\begin{equation}
\label{eq:abelianization}
p: \Jac(C_s)\isom H^{-1}(s) \longrightarrow
H^{-1}(0) = \mathcal{U}_C(n,d)
\end{equation}
is a covering morphism of degree $\delta$.

Let us now determine the Hamiltonian vector field
associated to each coordinate function of $H$. 
Take a point $(E,\phi)\in T^*\mathcal{U}_C(n,d)\cap H^{-1}(U_\reg)$.
The tangent space of the Higgs 
moduli space at this
point is given by
\begin{equation}
\label{eq:tangenttoHiggs}
T_{(E,\phi)}\mathcal{H}_C(n,d)=
H^1(C,\End(E))\dsum H^0(C,\End(E)\tensor K_C).
\end{equation}
Since the symplectic form $\omega=-d\eta$ is the exterior derivative 
of the tautological $1$-form $\eta$ of (\ref{eq:eta}) 
on $T^*\mathcal{U}_C(n,d)$,
the evaluation of $\omega$ 
at $(E,\phi)$  is given by
\begin{equation}
\label{eq:omegaevaluation}
\omega \big( (a_1,b_1), (a_2,b_2)\big)
= \langle a_1, b_2\rangle -\langle a_2,b_1\rangle,
\end{equation}
where 
$(a_1,b_1), (a_2,b_2)\in 
H^1(C,\End(E))\dsum H^0(C,\End(E)\tensor K_C)$, and
$$
\langle\cdot  ,\cdot\rangle : H^1(C,\End(E))\dsum H^0(C,\End(E)\tensor K_C)
\longrightarrow \mathbb{C}
$$
is the Serre duality pairing. This expression is the
standard Darboux form of the symplectic form $\omega$.
Choose an open
neighborhood $Y_1$ of $E$ in $\mathcal{U}_C(n,d)$
on which the cotangent bundle  $T^* \mathcal{U}_C(n,d)$
is trivial. Then $H$ is a polynomial
map
\begin{multline*}
H: Y_1\times H^0(C,\End(E)\tensor K_C)
\owns (E',\phi)\\
\longmapsto
s=\det(x-\phi)\in 
\bigoplus_{i=1} ^n H^0(C,K_C ^{\tensor i})
\end{multline*}
that depends only on the second factor. The differential
$dH_{(E,\phi)}$ at the point $(E,\phi)$ gives a linear
isomorphism
\begin{equation}
\label{eq:dH}
dH_{(E,\phi)}:H^0(C,\End(E)\tensor K_C)
\overset{\sim}{\longrightarrow}
\bigoplus_{i=1} ^n H^0(C,K_C ^{\tensor i}). 
\end{equation}
As to the first factor of the tangent space of (\ref{eq:tangenttoHiggs}),
we use the differential of the covering map $p$ of 
(\ref{eq:abelianization}):
\begin{equation}
\label{eq:dp}
dp_{(E,\phi)}: H^1(C_s, \mathcal{O}_{C_s})
\overset{\sim}{\longrightarrow}
H^1(C,\End(E)),
\end{equation}
where $s=H(E,\phi)$.
The dual of (\ref{eq:dH}), together with (\ref{eq:dp}), gives an 
isomorphism
\begin{multline}
\label{eq:H1OCs}
V_{GL}\overset{\text{def}}{=}
\bigoplus_{i=0} ^{n-1} H^1(C, K_C ^{-\tensor i})
= H^1(C,\pi_*\mathcal{O}_{C_s})\\
\overset{\sim}{\longrightarrow} H^1(C,\End(E))\isom 
H^1(C_s, \mathcal{O}_{C_s}).
\end{multline}
From (\ref{eq:dH}) and (\ref{eq:H1OCs}), we see that the tangent
space to the Higgs moduli  is given by
\begin{equation}
\label{eq:tangent2}
\begin{aligned}
T_{(E,\phi)} \mathcal{H}_C(n,d)
&\isom 
\bigoplus_{i=1} ^{n} H^1(C, K_C ^{-\tensor (i-1)})
\dsum
\bigoplus_{i=1} ^n H^0(C,K_C ^{\tensor i})\\
&=V_{GL}\dsum V^* _{GL}.
\end{aligned}
\end{equation}
Therefore, the symplectic form  has a
decomposition into $n$ pieces
$\omega =\omega_1+\omega_2+\cdots+\omega_n$, and in each
factor $H^1(C, K_C ^{-\tensor (i-1)})
\dsum
H^0(C,K_C ^{\tensor i})$, $\omega_i$
 takes the standard Darboux form
\begin{equation}
\label{eq:Darboux}
\omega_i \big( (a_1 ^i,b_1 ^i), (a_2 ^i,b_2 ^i)\big)
= \langle a_1 ^i, b_2 ^i\rangle_i -\langle a_2 ^i,b_1 ^i\rangle_i,
\end{equation}
where 
$(a_1 ^i,b_1 ^i), (a_2 ^i,b_2 ^i)\in 
H^1(C, K_C ^{-\tensor (i-1)})
\dsum
H^0(C,K_C ^{\tensor i})$,
and $\langle\cdot,\cdot\rangle_i$ is the duality pairing
of $H^1(C, K_C ^{-\tensor (i-1)})$ and 
$H^0(C,K_C ^{\tensor i})$.

Since $s=H(E,\phi)\in U_\reg$ is a regular value of $H$, there is an
open subset $Y_2\subset U_\reg\subset V^* _{GL}$ around
$s$ such that $H^{-1}(Y_2)$ is locally the 
product of the fiber $H^{-1}(s)$ and $Y_2$. By taking
$Y_1$ smaller if necessary, we thus obtain a local
product neighborhood
$$
p^{-1}(Y_1)\cap H^{-1}(Y_2) \isom Y_1\times Y_2 = Y
$$
of $(E,\phi)$ in $\mathcal{H}_C(n,d)$. 
By construction,
the projections to the first and the second factors
coincide with the projection  $p:T^* \mathcal{U}_C(n,d)\rightarrow
 \mathcal{U}_C(n,d)$  and the Hitchin map $H$:
\begin{equation*}
		\xymatrix{
		&Y\ar[dl]_{p} \ar[dr]^{H}&\\
		Y_1& & 	Y_2 .
		}
\end{equation*} 
The neighborhood $Y$ and these projections provide
the Darboux coordinate system for $\omega$, and its expression
(\ref{eq:omegaevaluation}, \ref{eq:Darboux}) globally holds on $Y$. In particular,
the Hamiltonian vector fields corresponding to the 
components of $H$ on $Y$ are constant vector fields
that are determined by elements of $H^1(C, K_C ^{-\tensor (i-1)})$
for each $i=1,\dots,n$. 
Let $(h_1,\dots,h_N)$ be a 
linear coordinate system of $V^* _{GL}$, where 
$N=n^2 (g-1) + 1$. Since $V^* _{GL} = \bigoplus_{i=1} ^n
H^0(C,K_C ^{\tensor i})$, each coordinate function is 
an element of $V_{GL}$:
\begin{equation}
\label{eq:hjs}
h_1,\dots,h_N\in V_{GL} = \bigoplus_{i=0} ^{n-1} H^1(C,
K_C ^{-\tensor i}).
\end{equation}
Since $H^0(C_s,\mathcal{O}_{C_s})=\mathbb{C}$, 
the Hitchin map $H:\mathcal{H}_C(n,d)\rightarrow V^* _{GL}$
is actually the pull back of the coordinate functions
$H^{-1}(h_1),\dots,H^{-1}(h_N)$ on $H^{-1}(U_\reg)
\subset \mathcal{H}_C(n,d)$. The identification of $V_{GL}$ 
as the first factor  of the tangent space in (\ref{eq:tangent2})
gives the Hamiltonian vector fields corresponding to the coordinate
components of the Hitchin map. 
We have therefore established
that the Hamiltonian vector fields are linear flows with respect
to the linear coordinate of the Jacobian $\Jac(C_s)$
determined by (\ref{eq:H1OCs}). Of course each of these linear flows
extends globally on $\Jac(C_s)$ since the tangent bundle of
a Jacobian is trivial and a local linear function extends 
globally as a linear function.

\begin{thm}[Hitchin \cite{H2}]
\label{thm:linear}
The Hamiltonian vector fields corresponding to 
the Hitchin fibration are linear Jacobian
flows on a generic fiber $H^{-1}(s)\isom \Jac(C_s)$.
\end{thm}

How does the construction of the \emph{spectral data}
$(\mathcal{L}_E, C_s)$ 
from a Higgs pair $(E,\phi)$ change
when we consider the Serre dual  $(E^*\tensor K_C,-\phi^*)$? 
To answer this question, first we note 
\begin{equation}
\label{eq:dualcharacteristic}
\begin{aligned}
\tr \wedge^i (\phi) =\tr \wedge^i(\phi^*) &\in H^0(C,\End(\wedge^i E)
\tensor K_C ^i)\\
&= H^0(C,\End(\wedge^i (E^*\tensor K_C))
\tensor K_C ^i).
\end{aligned}
\end{equation}
For $s = (s_1,s_2,\dots, s_n)\in V^* _{GL}$, we write
$s^* = (-s_1,s_2,\dots, (-1)^n s_n)$. By definition, the spectral
curves $C_s$ and $C_{s^*}$ are isomorphic. As divisors of
$T^*C$,  their isomorphism $C_s \isom \epsilon(C_{s^*})$
is induced by the involution of $T^*C$ 
\begin{equation}
\label{eq:involutionepsilon}
\epsilon : T^*C\owns (p,x)\longmapsto (p,-x)\in T^*C,
\end{equation}
where $p\in C$ and $x\in T_p ^*C$.

\begin{prop}[Hitchin \cite{H3}]
\label{prop:spectralforSerredual}
The spectral data 
$ (\mathcal{L}_{E^*\tensor K_C},C_{s^*})$
corresponding
to  the Serre dual $(E^*\tensor K_C,-\phi^*)$ 
of the Higgs pair $(E,\phi)$ is given by
$$
\begin{cases}
\mathcal{L}_{E^*\tensor K_C} = 
\epsilon^*(\mathcal{L}_E ^*\tensor 
K_{C_{s}})\\
C_{s^*} = \epsilon(C_{s}).
\end{cases}
$$
The degree of these isomorphic line bundles is
$-\deg(E) + (n^2+n)(g-1)$. 
\end{prop}

\begin{proof}
We use the exact sequence of \cite{BNR, H3} that characterizes 
$\mathcal{L}_E$:
$$
0\longrightarrow \mathcal{L}_E\tensor K_{C_s}^{-1}\tensor \pi^*K_C 
\longrightarrow \pi^*E\longrightarrow\pi^*(E\tensor K_C)\longrightarrow\mathcal{L}_E
\tensor \pi^*K_C\longrightarrow 0.
$$
The dual of this sequence is then
\begin{multline*}
0\longrightarrow \mathcal{L}_E ^*\tensor \pi^*K_C 
\longrightarrow \pi^*(E^* \tensor K_C)\longrightarrow\pi^*(E^*\tensor K_C ^{\tensor 2})\\
\longrightarrow\mathcal{L}_E ^* \tensor K_{C_s}
\tensor \pi^*K_C\longrightarrow 0.
\end{multline*}
Thus the spectral data of the Higgs pair $(E^* \tensor K_C,\phi^*)$ is 
$(\mathcal{L}_E ^* \tensor K_{C_s}, C_s)$. 
The involution $\epsilon$ comes in here when we consider the
Higgs pair $(E^* \tensor K_C,-\phi^*)$.
\end{proof}

\begin{prop}
\label{prop:SerreandSp}
Hitchin's integrable system for
the group $Sp_{2m}(\mathbb{C})$ is realized as 
the fixed-point-set of the Serre duality 
involution $(E,\phi)\mapsto (E^*\tensor K_C,-\phi^*)$
defined on the Higgs moduli space 
$\mathcal{H}_C(2m,2m(g-1))$.
\end{prop}

\begin{proof}
The fixed-point-set consists of Higgs pairs $(E,\phi)$ such
that $E\isom E^*\tensor K_C$ and $\phi = -\phi^*$.
Choose a square root of $K_C$ and define $F=E\tensor K_C ^{-1/2}$.
The Higgs field $\phi = -\phi^*$ naturally acts on this bundle,
and the pair $(F,\phi)$ forms the moduli space of $Sp_{2m}$-Higgs
bundles \cite{H3}. 
The characteristic coefficients satisfy the relation $s=s^*$, hence
\begin{equation}
\label{eq:Vsp}
s\in  V^* _{Sp}\overset{\text{def}}{=}\bigoplus_{i=1} ^m H^0(C,K_C ^{\tensor 2i}) .
\end{equation}
The spectral curve $C_s$ has a non-trivial involution
$\epsilon: C_s\rightarrow C_s$. 
From the exact sequence
\begin{multline*}
0\longrightarrow \mathcal{L}_E\tensor K_{C_s}^{-1}\tensor \pi^*K_C
^{1/2} 
\longrightarrow \pi^*(E \tensor K_C
^{-1/2})\longrightarrow\pi^*(E\tensor K_C ^{1/2})\\
\longrightarrow\mathcal{L}_E
\tensor \pi^*K_C ^{1/2}\longrightarrow 0,
\end{multline*}
we see that $\mathcal{L}_F \isom \mathcal{L}_E \tensor \pi^* K_C
^{-1/2}$ and $\mathcal{L}_{F^*} \isom \epsilon^* \mathcal{L}_F ^*
\tensor K_{C_s}\tensor \pi^* K_C ^{-1/2}$. If we define
$\mathcal{L}_0 = \mathcal{L}_F \tensor \pi^* K_C ^{-m +1/2}$  following
\cite{H3}, then $\mathcal{L}_0 ^* \isom \epsilon^* \mathcal{L}_0$.
\end{proof}

\section{Deformations of spectral curves}
\label{sect:deformation}

The Hitchin fibration (\ref{eq:Hitchinmap}) is 
a family of deformations of Jacobians, but  it is not 
 \emph{effective}. 
In this section we determine the natural effective family
associated with the Hitchin fibration.

An obvious action on $\mathcal{H}_C(n,d)$ that
preserves the spectral curves  is the scalar multiplication of Higgs fields
 $\phi\mapsto \lam\cdot\phi$ by $\lam\in \mathbb{C}^*$.
Although this $\mathbb{C}^*$-action is \emph{not}
symplectomorphic because it changes the symplectic form
$\omega\mapsto\lam\cdot\omega$, from the point of view
of constructing an effective family of Jacobians, we need to 
quotient it out. 
We note that the $\mathbb{C}^*$-action on $V^* _{GL}$ defined by
\begin{multline}
\label{eq:weighted}
V^* _{GL} = \bigoplus_{i=1} ^n H^0(C,K_C^{\tensor i})\owns
s=(s_1,s_2,\dots,s_n)\\
\longmapsto
\lam\cdot s=(\lam s_1,\lam^2 s_2,\dots, \lam^n s_n)\in V^* _{GL} 
\end{multline}
makes the Hitchin fibration  $\mathbb{C}^*$-equivariant:
\begin{equation}
\label{eq:scalarequivariance}
\begin{CD}
\mathcal{H}_C(n,d) @>{\qquad\lam\qquad}>> \mathcal{H}_C(n,d)\\
@V{H}VV @VV{H}V\\
 \bigoplus_{i=1} ^n H^0(C,K_C^{\tensor i})
 @>>{\text{diag}(\lam,\dots,\lam^n)}> 
 \bigoplus_{i=1} ^n H^0(C,K_C^{\tensor i}),
\end{CD}
\qquad \lam\in\mathbb{C}^*.
\end{equation}
Since $C_s$ is defined by the equation (\ref{eq:chareq}),
the natural $\mathbb{C}^*$-action on the cotangent
bundle $T^*C$ gives the isomorphism 
$\lam:C_s\rightarrow C_{\lam\cdot s}$ that commutes with
the projection $\pi:T^*C\rightarrow C$. The line bundles
$\mathcal{L}_{(E,\phi)}$ on $C_s$ and  $\mathcal{L}_{(E,\lam\phi)}$
on $C_{\lam\cdot s}$ corresponding to $E$ are related by this
isomorphism by
$$
\lam^* \mathcal{L}_{(E,\lam\phi)}\overset{\sim}{\longrightarrow}
\mathcal{L}_{(E,\phi)}.
$$

Another group action on the Higgs moduli space that leads to
trivial deformations of the spectral curves is the Jacobian action.
The  
$\Jac(C)$-action on $\mathcal{H}_C(n,d)$
is defined 
by 
$(E,\phi)\longmapsto (E\tensor L,\phi)$, where $L\in \Jac(C)
=\Pic^0(C)$
is a line bundle on $C$ of degree $0$. The Higgs field is preserved
in this action because 
$ (E\tensor L)^*\tensor (E\tensor L)=E^*
\tensor E$
 is unchanged, hence
 $$
 \phi\in 
 H^0(C,\End(E)\tensor K_C) = H^0(C,\End(E\tensor L)\tensor K_C).
 $$
  Thus the cotangent bundle $T^*\mathcal{U}_C(n,d)$
  is trivial along every orbit of the $\Jac(C)$-action on
  $\mathcal{U}_C(n,d)$, and each $\Jac(C)$-orbit in
  $\mathcal{H}_C(n,d)$ lies in the same fiber of the Hitchin
  fibration.   Let us identify the
$\Jac(C)$-action on a generic fiber $H^{-1}(s) \isom 
\Jac(C_s)$. 
The covering map 
$\pi:C_s \rightarrow C$ induces an injective  homomorphism
$
\pi^*: \Jac(C)\owns L \longmapsto \pi^*L\in  \Jac(C_s)
$.
This is injective because if $\pi^*L\isom \mathcal{O}_{C_s}$,
then by the projection formula we have
$$
\pi_*(\pi^*L) \isom \pi_*\mathcal{O}_{C_s}\tensor L\isom
\bigoplus_{i=0} ^{n-1} L\tensor K_C ^{\tensor -i},
$$
which has a nowhere vanishing section.
Hence $L\isom \mathcal{O}_C$. Take a point $(E,\phi)\in
H^{-1}(s)$ and let $\mathcal{L}_E$ be the corresponding line 
bundle on $C_s$. Since 
$\pi_*(\mathcal{L}_E\tensor \pi^*L) \isom E\tensor L$, 
the action of $\Jac(C)$ on $H^{-1}(s)\isom 
\Jac(C_s)$ is through the canonical action of the subgroup
$$
\Jac(C)\isom \pi^*\Jac(C)\subset \Jac(C_s)
$$ 
on $\Jac(C_s)$.

On the open subset 
$T^*\mathcal{U}_C(n,d)$
of  $\mathcal{H}_C(n,d)$,  the $\Jac(C)$-action is symplectomorphic
because it is
induced by the action on the base space $\mathcal{U}_C(n,d)$.
On the other open subset $H^{-1}(U_\reg)$ the action is also symplectomorphic 
because it preserves each fiber which is a Lagrangian. 
Since the symplectic form $\omega$ is defined by extending
the canonical form $\omega = -d\eta$ to $\mathcal{H}_C(n,d)$,
the $\Jac(C)$-action is globally holomorphic symplectomorphic
on $\mathcal{H}_C(n,d)$.
  This action is actually
 a Hamiltonian action and the first component
 of the Hitchin map
\begin{equation}
\label{eq:momentH1}
H_1:\mathcal{H}_C(n,d)\owns (E,\phi)\longmapsto
\tr(\phi) \in H^0(C,K_C)
\end{equation}
is the moment map. Note that 
$H^1(C,\mathcal{O}_C)$ is the Lie algebra of the Abelian
group $\Jac(C)$, and $H^0(C,K_C)$ is its dual Lie algebra. 
Since the infinitesimal action of $H^1(C,\mathcal{O}_C)$
on the Higgs moduli space defines a vector field
which is obtained by identifying $H^1(C,\mathcal{O}_C)$ with
 the first component of (\ref{eq:H1OCs}), 
the symplectic dual to this vector field is the map to the first 
component of second factor in (\ref{eq:tangent2}),
i.e, $H_1$.

We can therefore construct the symplectic quotient
$\mathcal{H}_C(n,d)/\!\!/\Jac(C)$, which we will do 
 in Section~\ref{sect:quotient}. 
Here our interest is
to determine an effective family of spectral curves.
The moment map $H_1$ of
(\ref{eq:momentH1})  being the trace of $\phi$, it  is natural  to 
define
\begin{equation}
\label{eq:Vsl}
V^* _{SL} = V^* _{SL_n(\mathbb{C})} = 
\bigoplus_{i=2} ^n H^0(C,K_C ^{\tensor i}) \subset
V^* _{GL}.
\end{equation}
This is a vector space of dimension $(n^2-1)(g-1)$. 
Now consider a partial projective space
\begin{equation}
\label{eq:PVSL}
\mathbb{P}(\mathcal{H}_C ^{SL}(n,d))
= \big(H^{-1}(V^* _{SL})\setminus H^{-1}(0)\big)\big/\mathbb{C}^* .
\end{equation}
This is no longer a holomorphic symplectic manifold, yet 
the Hitchin fibration naturally descends to a generically 
Jacobian fibration
\begin{equation}
\label{eq:PH}
PH: \mathbb{P}(\mathcal{H}_C ^{SL}(n,d))\longrightarrow
\mathbb{P}_w(V^* _{SL})
\end{equation}
over the weighted projective space of $V^* _{SL}$ defined 
by the restriction of (\ref{eq:weighted}) on $V^* _{SL}$.
We now claim

\begin{thm}[Effective Jacobian fibration]
\label{thm:effective}
The Jacobian fibration 
$$
PH: \mathbb{P}(\mathcal{H}_C ^{SL}(n,d))\longrightarrow
\mathbb{P}_w(V^* _{SL})
$$
 is generically effective.
\end{thm}

The rest of this section is devoted to the proof of this theorem.
The statement is equivalent to the claim 
that the family of deformations
of spectral curves parametrized by $\mathbb{P}_w(V^* _{SL})$
is effective. 
 Let $s\in V^* _{SL}$ be a point for which the spectral curve
 $C_s$ is  non-singular, and $\bar{s}$ the
 corresponding point of $\mathbb{P}_w(V^* _{SL})$.
 We wish to show that the Kodaira-Spencer map
 $$
 T_{\bar{s}}\,\mathbb{P}_w(V^* _{SL}) \longrightarrow
 H^1(C_s, \mathcal{T}C_s)
 $$
 is injective, where $\mathcal{T}X$ denotes the tangent sheaf of $X$.

The spectral curve $C_s$ of (\ref{eq:chareq})
is the divisor of $T^*C$ defined by
the section 
$$
\det(\tau-\phi)\in H^0(T^*C,\pi^*K_C ^{\tensor n}),
$$
where $\tau$ is the tautological section of $\pi^*K_C$ on $T^*C$.
Let $\mathcal{N}_s$ denote the normal sheaf of $C_s$ in $T^*C$.
Since $K_{T^*C} = \Lambda^2(T^*C) \isom \mathcal{O}_{T^*C}$
by the holomorphic symplectic form $\omega_C = -d\tau$, 
we have
$$
\mathcal{N}_s \isom K_{C_s}\isom \pi^*K_C ^{\tensor n}.
$$
From
\begin{equation*}
\begin{CD}
0@>>> \mathcal{T}C_s @>>> \mathcal{T} T^*C|_{C_s}
@>>> \mathcal{N}_s@>>>0,
\end{CD}
\end{equation*}
we obtain
\begin{equation}
\label{eq:KS}
\begin{CD}
0 @>>> H^0(C_s,\mathcal{T} T^*C)
@>{\iota}>> H^0(C_s,\mathcal{N}_s)
@>{\kappa}>> H^1(C_s,\mathcal{T}C_s)\\
@>>> H^1(C_s,\mathcal{T} T^*C)
@>>> H^1(C_s,\mathcal{N}_s).
\end{CD}
\end{equation}
Since 
$$
H^0(C_s,\mathcal{N}_s)\isom H^0(C,\pi_* K_{C_s})
\isom \bigoplus_{i=1} ^n H^0(C,K_C ^{\tensor i})= V^* _{GL},
$$
the homomorphism $\kappa$ is the Kodaira-Spencer map
for the deformations of spectral curves on $V^* _{GL}$.
We thus need to identify the image of $\iota$. We note that
the tangent sheaf and the cotangent sheaf are isomorphic on 
the total space of the cotangent bundle $T^*C$, i.e.,
$\mathcal{T}T^*C \isom \Lambda^1(T^*C)$.

\begin{prop}
We have the following isomorphisms:
\begin{equation}
\begin{aligned}
\label{eq:H0s}
&H^0(T^*C,\Lambda^0(T^*C)) \isom \mathbb{C}\\
&H^0(T^*C,\Lambda^1(T^*C))
\isom H^0(C,K_C)\dsum \mathbb{C}\cdot \tau \\
&H^0(T^*C,\Lambda^2(T^*C)) \isom \mathbb{C}.
\end{aligned}
\end{equation}
\end{prop}

\begin{proof}
The first isomorphism of (\ref{eq:H0s}) asserts that every globally
defined holomorphic function $f$ on $T^*C$ is a constant. 
A section $\sigma\in H^0(C,K_C)$  is
a map $\sigma:C\rightarrow T^*C$. Take an arbitrary pair
of points $(x,y)$ in $T^*C$. If they are not on the same
fiber, then there is a section $\sigma$ such that both $x$ and $y$
are on the image $\sigma(C)$. Since $f$ is constant on $\sigma(C)$,
$f(x) = f(y)$. If they are on the same fiber, then choose a point
$z\in T^*C$ not on this fiber and use the same argument.

Since $\Lambda^2(T^*C)\isom \mathcal{O}_{T^*C}$, the third
isomorphism follows from the first one.

The second isomorphism of (\ref{eq:H0s}) asserts that every
holomorphic $1$-form $\alpha$ on $T^*C$ is either the pull-back
of a holomorphic $1$-form on $C$ via $\pi:T^*C\rightarrow C$, the tautological 
$1$-form $\tau$, or a linear combination of them.
Let $\sigma:C\rightarrow T^*C$ be an arbitrary section. 
Since the normal sheaf of $\sigma(C)$ in $T^*C$ is 
$K_{\sigma(C)}$, we have on $\sigma(C)$
$$
\begin{CD}
0@>>>K_{\sigma(C)} ^{-1}
@>>>\Lambda^1(T^*C)\tensor \mathcal{O}_{\sigma(C)}
@>>>K_{\sigma(C)}
@>>>0.
\end{CD}
$$
Therefore, identifying $C\isom \sigma(C)$, we obtain
\begin{multline*}
0\longrightarrow H^0(C, \Lambda^1(T^*C)\tensor \mathcal{O}_C)
\longrightarrow H^0(C,K_C)
\overset{\kappa_0}{\longrightarrow}
H^1(C,K_C ^{-1})\\
\longrightarrow H^1(C, \Lambda^1(T^*C)\tensor \mathcal{O}_C)
\longrightarrow H^1(C,K_C)\longrightarrow 0.
\end{multline*}
Here $\kappa_0$ is the Kodaira-Spencer map assigning
a deformation of $C$ to a \emph{displacement} of
$C$ in $T^*C$ through a section of $K_C$. But since
$\sigma(C)$ is always isomorphic to $C$, we do not obtain
any deformation of $C$ in this way. Hence $\kappa_0$
is the $0$-map. Therefore, 
\begin{equation}
\label{eq:1formiso}
H^0(\sigma(C), \Lambda^1(T^*C)\tensor \mathcal{O}_{\sigma(C)})
\isom H^0(\sigma(C),K_{\sigma(C)})
\isom H^0(C,K_{C}).
\end{equation}
Now consider an exact sequence on
$T^*C$
\begin{multline*}
0\longrightarrow\Lambda^1(T^*C)\tensor \mathcal{O}_{T^*C}(-\sigma(C))
\longrightarrow\Lambda^1(T^*C)\\
\longrightarrow\Lambda^1(T^*C)\tensor \mathcal{O}_{\sigma(C)}
\longrightarrow0,
\end{multline*}
which produces
\begin{multline*}
0\longrightarrow H^0(T^*C, \Lambda^1(T^*C)\tensor 
\mathcal{O}_{T^*C}(-\sigma(C)))
\longrightarrow H^0(T^*C,\Lambda^1(T^*C))\\
\overset{r}{\longrightarrow}H^0(\sigma(C),\Lambda^1(T^*C)\tensor 
\mathcal{O}_{\sigma(C)}).
\end{multline*}
Because of (\ref{eq:1formiso}), the homomorphism $r$ is equal to
the pull-back
$$
r = \sigma^*: H^0(T^*C,\Lambda^1(T^*C))
\longrightarrow H^0(C,K_C)
$$
 by the section $\sigma:C\rightarrow T^*C$.
It is surjective because $\sigma^*\circ\pi^*=id_{H^0(C,K_C)}$. Therefore, we have a splitting exact sequence
\begin{multline}
\label{eq:splittingexact}
0\longrightarrow H^0(T^*C, \Lambda^1(T^*C)\tensor 
\mathcal{O}(-\sigma(C)))
\longrightarrow
H^0(T^*C,\Lambda^1(T^*C))\\
\overset{\sigma^*}{\longrightarrow}
H^0(C,K_C)\longrightarrow 0.
\end{multline}
The tautological $1$-form $\tau\in H^0(T^*C, \pi^*K_C)$ defines
$$
0\longrightarrow \pi^*K_C
\longrightarrow \Lambda^1(T^*C)
\overset{\wedge\tau}{\longrightarrow}
\Lambda^2(T^*C)\tensor \mathcal{O}_{T^*C}(-C)
\longrightarrow 0,
$$
noting that  $\tau$ vanishes along the divisor $C\subset T^*C$. 
Since 
$$
H^0(T^*C, \Lambda^2(T^*C)\tensor
\mathcal{O}_{T^*C}(-C)) \isom H^0(T^*C,\mathcal{O}_{T^*C}(-C))
=0,
$$
we obtain
$$
H^0(T^*C,\pi^*K_C) \isom H^0(T^*C,\Lambda^1(T^*C)).
$$
Take $\alpha\in H^0(T^*C, \Lambda^1(T^*C)\tensor 
\mathcal{O}_{T^*C}(-C)) \isom 
H^0(T^*C,\pi^*K_C (-C))$. Then
$$
\alpha/\tau\in H^0(T^*C,\pi^*K_C (-C)\tensor \pi^*K_C ^{-1}(C))
\isom H^0(T^*C,\mathcal{O}_{T^*C}) \isom\mathbb{C}.
$$
Therefore, $\alpha$ is a constant multiple of $\tau$, and we have
obtained
\begin{equation}
\label{eq:H0lambda1}
H^0(T^*C, \Lambda^1(T^*C)\tensor 
\mathcal{O}_{T^*C}(-C)) \isom \mathbb{C}.
\end{equation}
From (\ref{eq:splittingexact}) and (\ref{eq:H0lambda1}), we
conclude that
$$
H^0(T^*C,\Lambda^1(T^*C))\isom H^0(C,K_C)\dsum \mathbb{C}
\cdot \tau.
$$
\end{proof}

\begin{lem}
Let $s\in V^* _{GL}$ be a point such that 
the spectral curve $C_s$ is nonsingular.
Then we have an isomorphism
 \begin{equation}
 \label{eq:H0onCs}
H^0(C_s,\mathcal{T}T^*C)\isom  H^0(C_s,\Lambda^1(T^*C)) \isom
H^0(T^*C,\Lambda^1(T^*C)) .
\end{equation}
\end{lem}

\begin{proof}
The line bundle on $T^*C$ that corresponds to the
divisor $C_s\subset T^*C$ is $\pi^*K_C^{\tensor n}$. 
Thus 
$$
\mathcal{O}_{T^*C}(-C_s) \isom 
\pi^*K_C^{\tensor (-n)}\isom
\mathcal{O}_{T^*C}(-nC).
$$
As above, let us consider an exact sequence
$$
0\longrightarrow\Lambda^1(T^*C)\tensor \mathcal{O}(-nC)
\longrightarrow\Lambda^1(T^*C)
\longrightarrow\Lambda^1(T^*C)\tensor \mathcal{O}_{C_s}
\longrightarrow0
$$
and its cohomology sequence
\begin{multline*}
0 \longrightarrow
H^0(T^*C, \Lambda^1(T^*C)\tensor \mathcal{O}(-nC))
\longrightarrow
H^0(T^*C,\Lambda^1(T^*C))\\
\overset{q}{\longrightarrow}
H^0(C_s,\Lambda^1(T^*C)\tensor \mathcal{O}_{C_s}).
\end{multline*}
From (\ref{eq:H0lambda1}), we have
$$
H^0(T^*C, \Lambda^1(T^*C)\tensor \mathcal{O}_{T^*C}(-nC)) = 0
$$
for $n\ge 2$, hence $q$ is injective.

Take 
$\beta\in 
H^0(C_s,\Lambda^1(T^*C)\tensor \mathcal{O}_{C_s})$,
and extend it as a meromorphic $1$-form on $T^*C$. 
Since $\deg K_C = 2g-2>0$,
 every divisor of $T^*C$ intersects with the
$0$-section $C$, and since $C_s\sim nC$ as a divisor, 
it also intersects with $C_s$. If $D\subset T^*C$ is the pole divisor of $\beta$, then it
cannot intersect with $C_s$, hence $D=\emptyset$.
Therefore,  $q$ is surjective.
\end{proof}

Let us go back to the Kodaira-Spencer map (\ref{eq:KS}).
We now know from (\ref{eq:H0s}) and (\ref{eq:H0onCs}) that
\begin{equation}
\label{eq:KS2}
0\longrightarrow H^0(C,K_C)\dsum \mathbb{C}\cdot \tau
\overset{\iota}{\longrightarrow}
\bigoplus_{i=1} ^n H^0(C,K_C ^{\tensor i})
\overset{\kappa}{\longrightarrow} 
H^1(C_s,K_{C_s} ^{-1}).
\end{equation}
The $H^0(C,K_C)$-factor of the second term of
 (\ref{eq:KS2}) maps to the first
component of the third term via the injective homomorphism
$\iota$. The tautological $1$-form $\tau$ on $T^*C$ that appears
as the second factor of the second term is mapped to a diagonal 
ray in the third term.

To see this fact, we recall that the tangent 
and cotangent sheaves are isomorphic on $T^*C$ through
the symplectic form $\omega_C = -d\tau$. Let $v$ be the
vector field on 
$T^*C$ corresponding to $\tau$ through this isomorphism,
i.e., $\tau = \omega_C(\cdot,v)$. This vector field
$v$ represents the $\mathbb{C}^*$-action on 
$T^*C$ along fiber. In terms of a local coordinate
system, these correspondences are  clearly described.
Choose a local coordinate $z$ on the algebraic curve $C$
around a point $p\in C$,
and denote by $x$ the linear coordinate on $T_p ^*C$
with respect  to the basis $dz_p$. Then at the
point $(p,x dz_p)\in T^*C$ 
we have the following expressions:
$$
\begin{cases}
\tau = x \, dz\\
\omega_C=dz\wedge d x\\
v =x\, \frac{\partial}{\partial x}\,.
\end{cases}
$$
The $\lam\in\mathbb{C}^*$ action on $T^*C$ 
generated by  the vector fields $v$ produces a
displacement of $C_s\subset T^*C$ to $C_{\lam\cdot s}$,
which corresponds to the equivariant action of $\lam$ 
on $\mathcal{H}_C(n,d)$ as described in
(\ref{eq:scalarequivariance}). 
In terms of the holomorphic $1$-form $\tau$, its restriction
on $C_s$ gives an element  
\begin{multline*}
\tau|_{C_s}\in H^0(C_s,K_{C_s}) \isom
H^0(C,\pi_*\mathcal{O}_{C_s}\tensor K_C ^{\tensor n})
\isom
H^0(C,\dsum_{i=1} ^n K_C ^{\tensor i})\\
= \bigoplus_{i=1} ^n H^0(C,K_C ^{\tensor i}).
\end{multline*}
This is the image $\iota(\tau)$ of (\ref{eq:KS2}).

 Summing up, we have constructed an injective homomorphism
 \begin{equation}
 \label{eq:injectiveKS}
 V^* _{SL}/\mathbb{C} \isom
\big(\bigoplus_{i=1} ^n H^0(C,K_C ^{\tensor i})\big)\big/
\iota\big( H^0(C,K_C)\dsum \mathbb{C}\cdot \tau\big)
\overset{\bar{\kappa}}{\longrightarrow}
H^1(C_s,K_{C_s} ^{-1}).
\end{equation}
 The image of $\bar{\kappa}$
  represents the generic Kodaira-Spencer class of 
 the Hitchin fibration $PH$ of (\ref{eq:PH}).
 This completes the proof of Theorem~\ref{thm:effective}.

 \section{Symplectic quotient of the Higgs moduli space
 and Prym fibrations}
 \label{sect:quotient}

The Hamiltonian vector fields corresponding to the
coordinate components of the Hitchin map
$\mathcal{H}_C(n,d)\rightarrow V^* _{GL}$ are constant
Jacobian flows along each fiber of the map. Suppose we have
a direct sum decomposition of $V_{GL}$ into two Lie subalgebras
$$
V_{GL} \overset{\text{def}}{=} \bigoplus_{i=0} ^{n-1}
H^1(C,K_C ^{\tensor -i})
= \mathfrak{g}_1\dsum \mathfrak{g}_2.
$$
Then there are two possible ways to construct new algebraically
completely integrable Hamiltonian systems.
If the $\mathfrak{g}_1$-action on $\mathcal{H}_C(n,d)$
is integrated to a group action, then  it
 is Hamiltonian by definition and we can construct the symplectic
quotient. Or if the $\mathfrak{g}_2$-action 
is integrated to a group action instead, then we may find a
family of $\mathfrak{g}_2$-orbits in $\mathcal{H}_C(n,d)$
fibered over  the dual Lie algebra
$\mathfrak{g}_2 ^*$. 
An important difference between real symplectic geometry 
and holomorphic symplectic geometry is that in the latter case
integrations of the same Lie algebra may generate
different (non-isomorphic) Lie groups. Consequently,
the idea of symplectic quotient has to be generalized so
that we can allow  a \emph{family} of groups acting 
on a symplectic manifold.
The discovery of 
Hausel and Thaddeus in \cite{HT} is that the above
 two constructions lead
to \emph{mirror symmetric} pairs of Calabi-Yau spaces
in the sense of Strominger-Yau-Zaslow \cite{SYZ}.
 In this section we consider two cases, the
 $SL$-$PGL$ duality and the $Sp_{2m}$-$SO_{2m+1}$ duality.

The
 $SL$-$PGL$ duality comes from the decomposition
$$
V_{GL} = H^1(C,\mathcal{O}_C)
\oplus \left(\bigoplus_{i=i} ^{n-1}
H^1(C,K_C ^{\tensor -i})\right).
$$
Obviously,  the vector fields 
generated by the $H^1(C,\mathcal{O}_C)$-action
 are integrable to the $\Jac(C)$-action everywhere 
 on $\mathcal{H}_C(n,d)$. Therefore, we do have the usual
 symplectic quotient mod $\Jac(C)$. On the other hand,
 the integration of the other Lie algebra 
 $$
 V_{SL} \overset{\text{def}}{=}
 \bigoplus_{i=i} ^{n-1}
H^1(C,K_C ^{\tensor -i})
$$
produces different Lie groups, called \emph{Prym varieties},
 along each fiber of the
Hitchin fibration.

The spectral covering $\pi:C_s\rightarrow C$ induces
two group homomorphisms dual to one another,
 the pull-back $\pi^*$ and
the \emph{norm map} $\Nm_\pi$ defined by
\begin{equation}
\label{eq:normanddual}
\begin{aligned}
\pi^*: \;&\Jac(C)\owns L\longmapsto \pi^*L\in \Jac(C_s)\\
\Nm_\pi: \;&\Jac(C_s)\owns \mathcal{L}\longmapsto
\det(\pi_* \mathcal{L})\tensor \det(\pi_*\mathcal{O}_{C_s})^{-1}
\in \Jac(C).
\end{aligned}
\end{equation}
In terms of divisors the norm map can be defined alternatively by
$$
\Pic(C_s)\owns \sum_{p\in C_s} m(p)\cdot p
\longmapsto \sum_{p\in C_s} m(p)\cdot \pi(p)
\in\Jac(C).
$$
The \emph{Prym variety}  and the 
\emph{dual Prym variety} 
\begin{equation}
\label{eq:prymanddual}
\begin{aligned}
\Prym(C_s/C) &\overset{\text{def}}{=} \Ker(\Nm_\pi)\\
\Prym^*(C_s/C) &\overset{\text{def}}{=} \Jac(C_s)/\pi^*\Jac(C)
\end{aligned}
\end{equation}
constructed by using these 
homomorphisms
are Abelian varieties of dimension
$g(C_s)-g(C)$ and are  \emph{dual} to one another. 
The algebraically completely integrable Hamiltonian systems
with these Abelian fibrations, that are 
naturally constructed from $\mathcal{H}_C(n,d)$,
then become  SYZ-mirror symmetric.

We have shown in Section~\ref{sect:deformation}
  that the $\Jac(C)$-action on $\mathcal{H}_C(n,d)$ is 
Hamiltonian. So 
 we can define
the symplectic quotient
\begin{equation}
\label{eq:ph}
\mathcal{PH}_C(n,d)
\overset{\text{def}}{=}\mathcal{H}_C(n,d)/\!\!/\Jac(C) = H_1 ^{-1}(0)/\Jac(C).
\end{equation}
This is a symplectic space of
 dimension  $2(n^2-1)(g-1)$ modeled by
 the moduli space of stable principal $PGL_n(\mathbb{C})$-bundles
 on $C$ \cite{H1, H2}. 
Since
the $\Jac(C)$-action on $\mathcal{H}_C(n,d)$ preserves 
the Hitchin fibration, we have an induced Lagrangian
fibration
\begin{equation}
\label{eq:Hpgl}
H_{PGL}:\mathcal{PH}_C(n,d)\longrightarrow V^* _{SL}
=\bigoplus_{i=2} ^n H^0(C,K_C ^{\tensor i}).
\end{equation}
It's $0$-fiber  is $H_{PGL} ^{-1}(0) = \mathcal{U}_C(n,d)/\Jac(C)$.
Following \cite{MFK} we denote by $\mathcal{SU}_C(n,d)$ the
moduli space of stable vector bundles with a fixed determinant line
bundle. This is a fiber of the determinant map
\begin{equation}
\label{eq:det}
\mathcal{U}_C(n,d)\owns E\longmapsto \det E\in \Pic^d(C),
\end{equation}
and is independent of the choice of the value of the determinant. 
Note that (\ref{eq:det})  is a non-trivial fiber bundle. The
equivariant  $\Jac(C)$-action on (\ref{eq:det}) is given by
\begin{equation}
\label{eq:equivariantaction}
\begin{CD}
\mathcal{U}_C(n,d) @>{\tensor L}>> \mathcal{U}_C(n,d)\\
@V{\text{det}}VV @VV{\text{det}}V\\
\Pic^d(C) @>>{\tensor L^{\tensor n}}> \Pic^d(C)
\end{CD}
\qquad L\in \Jac(C).
\end{equation}
The isotropy subgroup of the $\Jac(C)$-action on 
$\mathcal{U}_C(n,d)$ is 
the group of $n$-torsion points
\begin{equation}
\label{eq:Jn}
J_n(C) \overset{\text{def}}{=}\{L\in \Jac(C)\,|\,
L^{\tensor n}=\mathcal{O}_C\} \isom
H^1(C,\mathbb{Z}/n\mathbb{Z}) ,
\end{equation}
since $E\tensor L\isom E$ implies $\det(E)\tensor L^{\tensor n}
\isom \det (E)$. 
Choose a reference line bundle $L_0\in \Pic^d(C)$ and 
consider a degree $n$ covering 
$$
\nu : \Pic^d(C)\owns L\tensor L_0\longmapsto 
L^{\tensor n}\tensor L_0\in\Pic^d(C) , \qquad L\in \Jac(C) .
$$
Then the pull-back bundle $\nu^*\mathcal{U}_C(n,d)$
on $\Pic^d(C)$ becomes trivial:
$$
\nu^*\mathcal{U}_C(n,d) = \Pic^d(C)\times \mathcal{SU}_C(n,d) .
$$
The quotient of this product by the diagonal action of 
$J_n(C)$ is the original moduli space:
\begin{equation}
\label{eq:Jnquotient1}
\big( \Pic^d(C)\times \mathcal{SU}_C(n,d)\big)\big/
J_n(C) \isom \mathcal{U}_C(n,d).
\end{equation}
It is now clear that 
$$
\mathcal{U}_C(n,d)/\Jac(C) \isom \mathcal{SU}_C(n,d)/
J_n(C) .
$$

What are the other fibers of (\ref{eq:Hpgl})?
Let $s\in V^* _{SL}\cap U_\reg$ be a point such that
$C_s$ is non-singular. 
We have already noted that  the covering map 
$\pi:C_s \rightarrow C$ induces an injective  homomorphism
$
\pi^*: \Jac(C)\owns L \longmapsto \pi^*L\in  \Jac(C_s).
$
Therefore, 
the fiber $H_{PGL} ^{-1}(s)$ is isomorphic to the 
dual Prym variety 
${\Prym}^*(C_s/ C)$.
Similarly to the equivariant action (\ref{eq:equivariantaction}),
we have 
\begin{equation}
\label{eq:equivariantjac}
\begin{CD}
\Jac(C_s) @>{\tensor L}>> \Jac(C_s)\\
@V{\Nm}VV @VV{\Nm}V\\
\Jac(C) @>>{\tensor L^{\tensor n}}> \Jac(C)
\end{CD}
\qquad L\in \Jac(C).
\end{equation}
By the same argument  of (\ref{eq:Jnquotient1}), we obtain
\begin{equation}
\label{eq:Jnquotient2}
\big(\Prym(C_s/C)\times \Jac(C)\big)\big/J_n(C) \isom \Jac(C_s).
\end{equation}
From (\ref{eq:prymanddual}) and (\ref{eq:Jnquotient2}), it follows
that
$$
{\Prym}^*(C_s/ C) =\Prym(C_s/ C) /J_n(C).
$$
We have thus established

\begin{thm}[\cite{H2,H3}]
The natural fibration 
$$
H_{PGL}:\mathcal{PH}_C(n,d)\longrightarrow V^* _{SL}
=\bigoplus_{i=2} ^n H^0(C,K_C ^{\tensor i})
$$
of  \emph{(\ref{eq:Hpgl})} is a Lagrangian dual Prym
fibration with respect to the canonical holomorphic
symplectic form $\bar{\omega}$ on $\mathcal{PH}_C(n,d)$.
\end{thm}

We recall that the Higgs moduli space $\mathcal{H}_C(n,d)$
contains the cotangent bundle $T^*\mathcal{U}_C(n,d)$
as an open dense  subspace, and that the holomorphic 
symplectic form $\omega$ is the canonical symplectic form
on this cotangent bundle. Similarly, we can show the following

\begin{prop}
\label{prop:PGLcotangent}
The symplectic form $\bar{\omega}$ on 
$\mathcal{PH}_C(n,d)$ given by the symplectic quotient
is the canonical cotangent symplectic form on the cotangent
bundle
$$
T^*\big(\mathcal{SU}_C(n,d)/J_n(C)\big)\subset 
\mathcal{PH}_C(n,d).
$$
\end{prop}

\begin{proof}
Let $E$ be a stable vector bundle on $C$. The exact sequence
$$
\begin{CD}
0@>>> \mathcal{O}_C
@>>>\End(E)
@>>>\mathcal{Q}=\End(E)/\mathcal{O}_C
@>>>0
\end{CD}
$$
induces a cohomology sequence
$$
\begin{CD}
0@>>>H^1(C,\mathcal{O}_C)
@>>>H^1(C,\End(E)) .
\end{CD}
$$
Therefore, 
$$
T_E \big(\mathcal{SU}_C(n,d)/J_n(C)\big)
\isom 
H^1(C,\End(E))/H^1(C,\mathcal{O}_C).
$$
Dualizing the situation, we have
$$
0\longrightarrow\End_0(E)\tensor K_C
\longrightarrow\End(E)\tensor K_C
\overset{\tr}{\longrightarrow}
K_C
\longrightarrow 0,
$$
where $\End_0(E)$ is the sheaf of traceless endomorphisms 
of $E$. We then have
\begin{multline*}
0\longrightarrow
H^0(C,\End_0(E)\tensor K_C)
\longrightarrow
H^0(C,\End(E)\tensor K_C)\\
\overset{\tr}{\longrightarrow}
H^0(C,K_C)
\longrightarrow 0.
\end{multline*}
The trace homomorphism is globally surjective
because $\mathcal{O}_C$ is contained in $\End(E)$. 
Hence
$$
T^* _E \big(\mathcal{SU}_C(n,d)/J_n(C)\big)
\isom 
H^0(C,\End_0(E)\tensor K_C).
$$
Since the moment map $H_1$ of (\ref{eq:momentH1}) is 
the trace map, we conclude that the symplectic form 
$\bar{\omega}$ on the symplectic quotient 
$\mathcal{PH}_C(n,d)$ is the canonical cotangent symplectic
form on 
the cotangent bundle
$$
T^*  \big(\mathcal{SU}_C(n,d)/J_n(C)\big).$$
\end{proof}

The other reduction of $\mathcal{H}_C(n,d)$ 
consisting of the $V_{SL}$-orbits is
 the moduli 
space $\mathcal{SH}_C(n,d)$
of stable Higgs bundles $(E,\phi)$ 
 with a fixed determinant
$\det(E) \isom L$ and  traceless Higgs fields
$$
\phi\in  H^0(C,\End_0(E)\tensor K_C).
$$
This moduli space is modeled by the moduli space of stable principal
$SL_n(\mathbb{C})$-bundles on $C$ and has dimension 
$2(n^2-1)(g-1)$.
The cotangent bundle $T^* \mathcal{SU}_C(n,d)$
is an open dense subspace of 
$\mathcal{SH}_C(n,d)$.
Since $\tr(\phi)=0$, the Hitchin fibration 
$H$ of (\ref{eq:Hitchinmap}) naturally restricts to
\begin{equation}
\label{eq:Hsl}
H_{SL}:\mathcal{SH}_C(n,d)
\longrightarrow
V^* _{SL} = \bigoplus_{i=2} ^n H^0(C,K_C ^{\tensor i}).
\end{equation}
The $0$-fiber is $H_{SL} ^{-1}(0) = \mathcal{SH}_C(n,d)$.
For a generic $s\in V^* _{SL}$ such that $C_s$ is 
 non-singular, the fiber $H_{SL} ^{-1}(s)$ is the subset 
 of $H^{-1}(s) \isom \Jac(C_s)$ consisting of Higgs bundles
 $(E,\phi)$ such that $\det(E) = L$ is fixed and 
 $\det(x-\phi) =s$. Therefore,  $H_{SL} ^{-1}(s) \isom
\Prym(C_s/C)$.

By comparing $\mathcal{PH}_C(n,d)$ and $\mathcal{SH}_C(n,d)$,
we find that
\begin{equation}
\label{eq:PHandSH}
\mathcal{SH}_C(n,d)\big/J_n(C) \isom \mathcal{PH}_C(n,d),
\end{equation}
where the $J_n(C)$-action on $\mathcal{SH}_C(n,d)$ is defined
by $E\mapsto E\tensor L$ for  $L\in J_n(C)$. 
Since $J_n(C)$ is a finite group and $\mathcal{PH}_C(n,d)$
is a holomorphic symplectic variety, we can define a 
holomorphic symplectic form $\hat{\omega}$ on 
$\mathcal{SH}_C(n,d)$ via the pull-back of the projection
$$
\mathcal{SH}_C(n,d)\longrightarrow 
\mathcal{SH}_C(n,d)\big/J_n(C) .
$$
Obviously $\hat{\omega}$ agrees with the canonical 
cotangent symplectic form on $T^*\mathcal{SU}_C(n,d)$. 
Therefore,

\begin{thm}[\cite{H2,H3}]
The fibration 
$$
H_{SL}:\mathcal{SH}_C(n,d)
\longrightarrow
V^* _{SL} = \bigoplus_{i=2} ^n H^0(C,K_C ^{\tensor i})
$$
is a Lagrangian Prym fibration with respect to the
canonical holomorphic symplectic form
$\hat{\omega}$ on $\mathcal{SH}_C(n,d)$.
\end{thm}

Hausel and Thaddeus \cite{HT}
shows that

\begin{thm}[Theorem~(3.7) in \cite{HT}]
\label{thm:mirror}
The two fibrations
\begin{equation*}
		\xymatrix{
		\mathcal{SH}_C(n,d)\ar[dr]^{H_{SL}}
		& & 	\mathcal{PH}_C(n,d)\ar[dl]_{H_{PGL}}\\
		& V^* _{SL} &
		}
\end{equation*} 
are mirror symmetric in the sense of Strominger-Yau-Zaslow
 \cite{SYZ}.
\end{thm}

The effectiveness of the family of Prym and dual Prym varieties
can be established by the same method of
Section~\ref{sect:deformation}. Let us define
 partial projective moduli spaces 
$$
\mathbb{P}(\mathcal{PH}_C(n,d))\overset{\text{def}}{=}
\big(\mathcal{PH}_C(n,d)\setminus H_{PGL}^{-1}(0)\big)\big/
\mathbb{C}^*
$$
and 
$$
\mathbb{P}(\mathcal{SH}_C(n,d))\overset{\text{def}}{=}
\big(\mathcal{SH}_C(n,d)\setminus H_{SL}^{-1}(0)\big)\big/
\mathbb{C}^*.
$$
 The induced 
Hitchin maps are denoted by
\begin{equation}
\label{eq:projectivepryms}
\begin{aligned}
PH_{PGL} : \;
&\mathbb{P}(\mathcal{PH}_C(n,d))
\longrightarrow \mathbb{P}_w(V^* _{SL})\\
PH_{SL} : \;
&\mathbb{P}(\mathcal{SH}_C(n,d))
\longrightarrow \mathbb{P}_w(V^* _{SL}).
\end{aligned}
\end{equation}
We have the following

\begin{thm}
\label{thm:dualprymfibration}
The Prym and dual Prym fibrations 
$$
\begin{aligned}
PH_{SL} : \;
&\mathbb{P}(\mathcal{SH}_C(n,d))
\longrightarrow \mathbb{P}_w(V^* _{SL})\\
PH_{PGL} : \;
&\mathbb{P}(\mathcal{PH}_C(n,d))
\longrightarrow \mathbb{P}_w(V^* _{SL})
\end{aligned}$$
are generically effective.
\end{thm}

For the case of 
$Sp_{2m}(\mathbb{C})$-Hitchin systems, we
 consider Lie subalgebras
 \begin{equation}
 \label{eq:Vspandg}
 \begin{aligned}
  \mathfrak{g}=
&\bigoplus_{i=0} ^{m-1}H^1\!\left(C, K_C ^{\tensor -2i}\right)\\
 V_{Sp} =
   &\bigoplus_{i=0} ^{m-1} H^1\!\left(C,K_C ^{\tensor -2i-1}\right),
  \end{aligned}
 \end{equation}
 and a direct sum decomposition 
\begin{equation}
\label{eq:Spdsum}
V_{GL_{2m}} = 
\mathfrak{g}
\dsum V_{Sp}.
\end{equation}
This time the Hamiltonian flows on $\mathcal{H}_C(n,d)$
generated by elements of $ \mathfrak{g}$
 do not form a group
action. However, if we restrict our attention to 
points of $U_\reg \cap
V^* _{Sp}$ as in (\ref{eq:Vsp}), then the integral of 
the  Lie algebra action
becomes a group action.
Recall that the spectral curve $C_s$ has an involution
induced by $\epsilon$ of (\ref{eq:involutionepsilon}). We denote by
\begin{equation}
\label{eq:degree2cover}
r:C_s\rightarrow C_s '=C_s/\langle \epsilon\rangle
\end{equation}
the natural
projection. It is ramified at the intersection of $C_s$ with the 
$0$-section of
$T^*C$, which is the divisor of $\pi^*K_C$
 on $C_s$  of degree
$4m(g-1)$. Therefore, we find the genus of $C' _s$ by the
Riemann-Hurwitz formula:
\begin{equation}
\label{eq:genusofC'}
g(C' _s) = m(2m-1)(g-1) + 1 = \dim \mathfrak{g}.
\end{equation}
Since $H^0(C' _s,r_*\mathcal{O}_{C_s})$ has a nowhere
vanishing section, we have
\begin{equation}
\label{eq:rdirect}
\begin{CD}
0@>>>\mathcal{O}_{C_s '} @>>>
r_* \mathcal{O}_{C_s}@>>>
 N^{-1}@>>>0
\end{CD}
\end{equation}
with  a line bundle $N$ on $C' _s$ of degree $2m(g-1)$. 
Note that 
$$
N^{\tensor 2} = \det(r_*\mathcal{O}_{C_s})^{\tensor -2} \isom
\Nm_r(\pi^*K_C),
$$
hence $N$ is a square root of the \emph{branch divisor}
$\Nm_r(\pi^*K_C)$ of the covering $r$. 
The exact sequence (\ref{eq:rdirect}) gives
\begin{equation}
\label{eq:Sph1decomp}
H^1(C_s,\mathcal{O}_{C_s})
\isom 
H^1(C' _s,\mathcal{O}_{C' _s}) \dsum
H^1(C' _s,N^{-1}),
\end{equation}
and the projection to the first factor is the differential of  
the norm map
\begin{equation*}
\Nm_r:\Jac(C_s) \owns \mathcal{L}\longmapsto
\det(r_* \mathcal{L})\tensor \det(r_* \mathcal{O}_{C_s})^{-1}
\in\Jac(C' _s).
\end{equation*}

The construction of the $Sp$-Hitchin system 
of Proposition~\ref{prop:SerreandSp} is to reduce
 the $GL$-Hitchin system $\mathcal{H}_C(2m,2m(g-1))$
 by finding the right fibration of groups. Along 
 the fixed-point-set of the Serre duality on this Higgs moduli
 space, the action of 
 $V_{Sp}$
 generates the Prym fibration $\{\Prym(C_s/C' _s)\}_{s\in V^* _{Sp}}$
 since the condition 
 $\mathcal{L}_0 ^* \isom \epsilon^* \mathcal{L}_0$
on $C_s$ 
is the same as $\mathcal{L}_0\in \Prym(C_s/C_s ')$.
Comparing (\ref{eq:Spdsum}) and (\ref{eq:Sph1decomp}), 
we have
$$
H^1(C' _s,\mathcal{O}_{C' _s}) \isom 
\mathfrak{g} \qquad
\text{and}
\qquad
H^1(C' _s,N^{-1}) \isom V_{Sp}.
$$

The \emph{dual} fibration is the result of a kind of 
symplectic quotient of $\mathcal{H}_C(2m,2m(g-1))$
by $\mathfrak{g}$. 
In this quotient we restrict the Hitchin fibration 
to the $0$-fiber of 
$$
\mathfrak{g}^* =
\bigoplus_{i=1} ^{m}H^0\!\left(C, K_C ^{\tensor 2i-1}\right),
$$
and then take the quotient of each fiber by the Lie
group $\Jac(C' _s)$ of $\mathfrak{g}$. 
The result is the dual Prym fibration with a fiber
$$
\Prym^*(Cs/C' _s) = \Jac(C_s)/\Jac(C' _s) =
\Prym(Cs/C' _s) /J_2(C' _s)
$$
over each $s\in V^* _{Sp}$, where $J_2(C' _s)$ denotes
the group of $2$-torsion points of $\Jac(C' _s)$.

\section{Hitchin's integrable systems and the KP equations}
\label{sect:KP}

The partial projective Hitchin fibration 
$$PH:\mathbb{P}(\mathcal{H}_C ^{SL}(n,d))\rightarrow 
\mathbb{P}_w(V^* _{SL})$$ is a generically effective
Jacobian fibration. 
In this section we embed 
$\mathbb{P}(\mathcal{H}_C ^{SL}(n,d))$ into a 
quotient of the Sato Grassmannian. 
There is also a natural embedding of $\mathcal{H}_C(n,n(g-1))$
into a \emph{relative} Grassmannian of \cite{AMP, DM, PM, Q}.
We show that the linear Jacobian
flows on  the Hitchin integrable
system $(\mathcal{H}_C(n,n(g-1)), \omega, H)$
  are exactly the KP equations 
 on the
Grassmannian via this second  embedding.

Following Quandt \cite{Q}, we define

\begin{Def}[Sato Grassmannian] 
\label{def:Sato}
Let $n$ be a positive integer.
The Sato Grassmannian $Gr_n$ is a functor from the category of
schemes to the small category of sets. It assigns to every 
scheme $S$ a set $Gr_n(S)$ consisting of quasi-coherent 
$\mathcal{O}_S$-submodules $W$ of $\mathcal{O}_S((z)) ^{\dsum n}$
such that both the kernel and the cokernel of 
the natural homomorphism 
\begin{equation}
\label{eq:fredholm}
0\longrightarrow\Ker(\rgam_W)\longrightarrow
W
\overset{\rgam_W}\longrightarrow \mathcal{O}_S((z)) ^{\dsum n}\big/ \mathcal{O}_S[[z]] ^{\dsum n}
\longrightarrow\Coker(\rgam_W)\longrightarrow 0
\end{equation}
are coherent $\mathcal{O}_S$ modules. We refer to this condition
simply the \emph{Fredholm condition}. Here we denote by 
$\mathcal{O}_S[[z]]$ the ring of formal power series in $z$ with
coefficients in $\mathcal{O}_S$, and
$\mathcal{O}_S((z)) = \mathcal{O}_S[[z]] +\mathcal{O}_S[z^{-1}]$.
\end{Def}

\begin{rem}
A more general  and powerful 
theory of relative Sato Grassmannians has been
recently established by Plaza Mart\'in \cite{PM}. 
It would be an interesting project to study possible 
relations between \cite{PM} and the geometric Langlands
correspondence.
\end{rem}

\begin{rem}
\label{rem:index}
The Grassmannian $Gr_n(\mathbb{C})$ defined over
the point scheme $\Spec(\mathbb{C})$ has the
structure of an infinite-dimensional pro-ind scheme over
$\mathbb{C}$. If $S$ is irreducible, then $Gr_n(S)$ 
is a disjoint union of an infinite number of 
components \emph{indexed} 
by the Grothendieck group $K(S)$:
\begin{equation}
\label{eq:index}
\index: Gr_n(S)\owns W \longmapsto \text{index}(\rgam_W)
\overset{\text{def}}{=} \Ker(\rgam_W)-\Coker(\rgam_W)\in K(S).
\end{equation}
\end{rem}

\begin{rem}
\label{rem:bigcell}
The \emph{big cell} is the open subscheme of 
the index $0$ piece of the Grassmannian $Gr_n ^0(S)$
consisting of $W$'s such that $\rgam_W$ is an isomorphism.
This is the stage where 
the Lax and Zakharov-Shabat formalisms of integrable nonlinear
partial differential equations, such as KP, KdV, and many other
equations, are interpreted as infinite-dimensional dynamical systems
\cite{SS, SW}. The fundamental result is the identification 
(\ref{eq:groupofpdo}) of the big cell with the
group of $0$-th order monic pseudo-differential operators 
 \cite{LM2, Q, SS, SW}.
\end{rem}

\begin{rem}
\label{rem:topology}
The above $W$ is a closed subset of $\mathcal{O}_S((z))^{\dsum n}$ with respect to 
the topology defined by the filtration
$$
\cdots\supset z^{\nu-1}\cdot \mathcal{O}_S[[z]]^{\dsum n}\supset 
z^{\nu}\cdot \mathcal{O}_S[[z]]^{\dsum n}\supset 
z^{\nu+1}\cdot\mathcal{O}_S[[z]]^{\dsum n}\supset \cdots .
$$
\end{rem}

One of the consequences of 
the  Fredholm condition is that
\begin{equation}
\label{eq:largenu}
W\cap z^\nu\cdot\mathcal{O}_S[[z]] ^{\dsum n}= 0 \qquad \text{for } \nu>> 0.
\end{equation}
Consider the group 
\begin{equation}
\label{eq:groupGamma}
\Gamma_n(S) = \mathcal{O}_S ^\times \cdot
\begin{bmatrix}
1\\
&1\\
&&\ddots\\
&&&1
\end{bmatrix}
+z\cdot
\begin{bmatrix}
\mathcal{O}_S[[z]]\\
&\mathcal{O}_S[[z]]\\
&&\ddots\\
&&&\mathcal{O}_S[[z]]
\end{bmatrix}
\end{equation}
 consisting of $n\times n$ invertible diagonal
matrices with entries in the ring of formal power series whose
constant term is scalar diagonal.
It acts on $Gr_n(S)$ by left-multiplication without fixed points. Indeed, let 
$W\in Gr_n(S)$ and
$c+\rgam\in   \mathcal{O}_S ^\times \cdot I_n +z\cdot \mathcal{O}_S[[z]]
^{\dsum n}$
satisfy  that
$W=(c+\rgam)\cdot W$. This means that $\rgam\cdot w\in W$
for every $w\in W$ since $W$ is an $\mathcal{O}_S$-module.
Since $\rgam\in z\cdot \mathcal{O}_S[[z]] ^{\dsum n}$, it 
then contradicts to (\ref{eq:largenu}) unless 
$\rgam=0$. The \emph{quotient Grassmannian} 
\begin{equation}
\label{eq:quotientgrassmannian}
Z_n \overset{\text{def}}{=}Gr_n/\Gamma_n
\end{equation}
is a smooth infinite-dimensional scheme.
We denote by $\overline{W}$ the point of $Z_n$
corresponding to $W\in Gr_n$.

The tangent space $T_W Gr_n$ of the Sato Grassmannian 
at  $W$ is given  by  the space of continuous
$\mathcal{O}_S$-homomorphisms
$$
T_W Gr_n= \Hom_{\mathcal{O}_S}\big(W,\mathcal{O}_S((z))
 ^{\dsum n}/W\big).
$$
The tangent space to $Z_n$ is then given by
\begin{equation}
\label{eq:tangenttoZ}
T_{\overline{W}}Z_n \isom \Hom_{\mathcal{O}_S}\big(W,\mathcal{O}_S((z)) ^{\dsum n}\big/(
\Gamma_n\cdot W)\big).
\end{equation}
This expression does not depend on the choice of the 
lift $W\in Gr_n$ of $\overline{W}\in Z_n$.

Every element $a\in \mathcal{O}_S((z)) ^{\dsum n}$ defines a
homomorphism
$$
KP(a)_W: W\owns w\longmapsto \overline{a\cdot w}\in
\mathcal{O}_S((z))  ^{\dsum n}\big/(\Gamma_n \cdot W)
$$
through the left multiplication as a diagonal matrix,
which in turn determines a
global vector field
$$
\mathcal{O}_S((z)) ^{\dsum n}\owns a\longmapsto
KP(a)\in H^0(Z_n,TZ_n).
$$
We call $KP(a)$ the $n$-component 
\emph{KP flow} associated with $a$. 
As explained in \cite{LM2, M1984, M1990, Q}, the quotient of the
index zero Grassmannian $Z_n ^0$ is naturally identified with 
the set of \emph{Lax operators} (\ref{eq:Laxoperator}), and the action of 
$\mathcal{O}_S((z)) ^{\dsum n}$ on a Lax operator 
is written as an infinite
system of nonlinear partial differential equations called 
\emph{Lax equations}. This system 
for the case of $S=\Spec(\mathbb{C})$  is
the  $n$-component \emph{Kadomtsev-Petviashvili hierarchy}
\cite{LM2}.

Associated to the Hitchin fibration 
$H:\mathcal{H}_C(n,d)\rightarrow 
V^* _{GL}$
 we have a family of spectral curves:
\begin{equation}
\label{eq:familyofspectral}
		\xymatrix{
		\mathfrak{C}_{V^* _{GL}}(n,d)\;\ar[dr]_{F}
		 \ar[rr]^{\text{inclusion}\quad}
		&& 	\;\; T^*C\times V^* _{GL}\ar[dl]^{p_2}\\
		& \quad V^* _{GL} &
		}
\end{equation} 
Here the fiber $F^{-1}(s)$ of $s\in V^* _{GL}$ is the
spectral curve $C_s\subset T^*C\times \{s\}$, and $p_2$ is 
the projection to the second factor.
We note that the ramification points of the covering $\pi:C_s
\rightarrow C$
are determined by the resultant of the defining equation
\begin{equation}
\label{eq:poly}
x^n + s_1 x^{n-1} + \cdots + s_n = 0
\end{equation}
of $C_s$ and its derivative 
\begin{equation}
\label{eq:derivative}
nx^{n-1} + (n-1) s_1 x^{n-1} +\cdots + s_{n-1} = 0.
\end{equation}
For every $s\in V^* _{GL}$, we denote by $\Res(s)$ this resultant. 
The Sylvester matrix of these polynomials (\ref{eq:poly})
and (\ref{eq:derivative}) show that 
$$
\Res(s) \in H^0(C,K_C ^{\tensor n(n-1)}).
$$
Since the linear system $|K_C|$ is base-point-free, for every choice of 
$p\in C$, the subset 
\begin{equation}
\label{eq:unramified}
U_p= \{s\in V^* _{GL}\;|\; C_s \text{ is non-singular and }
\pi:C_s\rightarrow C\text{ is unramified at } p\}
\end{equation}
 is Zariski open in $V^* _{GL}$.

\begin{thm}
\label{thm:projectiveembedding}
There is a rational map 
$$
\mu:\mathbb{P}(\mathcal{H}_C ^{SL}(n,d))
\longrightarrow Z_n ^{d-n(g-1)}(\mathbb{C})
$$ 
of 
the partial projective moduli space of Higgs bundles
into
the quotient Grassmannian  of
index $d-n(g-1)$.
 This map is generically injective.
At a general point of the image of the 
embedding the $n$-component
KP flows defined on $Z_n(\mathbb{C})$ are tangent to the
Hitchin fibration 
$PH: \mathbb{P}(\mathcal{H}_C ^{SL}(n,d))
\longrightarrow \mathbb{P}_w(V^* _{SL})$.
\end{thm}

\begin{proof}
A point of the partial projective Higgs moduli space
represents an isomorphism class of spectral data
$(\pi:C_s\rightarrow C, \mathcal{L})$,
where $\mathcal{L}$ is a line bundle
 on $C_s$ of degree $d+(n^2-n)(g-1)$.  Choose a
point $p\in C$, a coordinate $z$ of the formal completion $\hat{C}_p$
of $C$ at $p$, 
and $s\in U_p\cap V^* _{SL}$ so that $C_s$ is non-singular
and $\pi$ is unramified over $p$. The formal coordinate $z$ defines 
an identification $\mathcal{O}_{\hat{C}_p} = \mathbb{C}[[z]]$. 
We
also choose a local trivialization of $\mathcal{L}$ around
$\pi^{-1}(p)$, i.e., an isomorphism
\begin{equation}
\label{eq:Ltrivial}
\beta: \mathcal{L}|_{\hat{C}_{s,\pi^{-1}(p)}}
\overset{\sim}{\longrightarrow} \mathcal{O}_{\hat{C}_{s,\pi^{-1}(p)}}
= \mathbb{C}[[z]]^{\dsum n}.
\end{equation}
Since the formal completion $\hat{C}_{s,\pi^{-1}(p)}$ is the disjoint
union of $n$ copies of $\hat{C}_p$, the formal coordinate $z$ also 
defines an identification $\mathcal{O}_{\hat{C}_{s,\pi^{-1}(p)}}
= \mathbb{C}[[z]]^{\dsum n}$. 

Now define
$$
W = \beta(H^0(C_s\setminus \pi^{-1}(p),\mathcal{L}))
\subset \mathbb{C}((z))^{\dsum n},
$$
which is the set of meromorphic sections of $\mathcal{L}$ that are 
holomorphic on $C_s\setminus \pi^{-1}(p)$ and have finite poles
at $\pi^{-1}(p)$. Since we have
$$
\begin{cases}
\Ker(\rgam_W) \isom H^0(C_s,\mathcal{L})\\
\Coker(\rgam_W) \isom H^1(C_s,\mathcal{L}),
\end{cases}
$$
$W$ is a point of the Grassmannian $Gr_n ^{d-n(g-1)}$
of index $d-n(g-1)$. The different choice of the local trivialization
$\beta'$ in (\ref{eq:Ltrivial}) leads to an element 
$\beta'\circ\beta^{-1}\in \Gamma_n$. Therefore, the point 
$\overline{W}\in Z_n ^{d-n(g-1)}(\mathbb{C})$ is uniquely 
determined by $(\pi:C_s\rightarrow C, p,z, \mathcal{L})$. 
Conversely this set of geometric data is uniquely determined
by $\overline{W}$ (see Section~5 of \cite{LM2}). Thus
the rational map $\mu$ is generically one-to-one. 
Notice that the $\Gamma_n$ action on the 
Grassmannian is inessential from the geometric point of view
because it simply changes the local trivialization of (\ref{eq:Ltrivial}).

The tangent space at a general point of 
$\mathbb{P}(\mathcal{H}_C ^{SL}(n,d))$
is $$H^1(C_s,\mathcal{O}_{C_s}) \dsum V^* _{SL}/\mathbb{C}.
$$
Since the Kodaira-Spencer map (\ref{eq:injectiveKS}) is injective,
it suffices to show injectivity of  the natural map
\begin{multline*}
d\mu : H^1(C_s,\mathcal{O}_{C_s}) \dsum H^1(C_s,K_{C_s} ^{-1})
\longrightarrow \Hom(W,\mathbb{C}((z))^{\dsum n}/W)\\
=T_W Gr_n(\mathbb{C}),
\end{multline*}
which is 
induced by the local trivialization of $\mathcal{O}_{C_s}$ and 
$K_C$ coming from the choice of the local coordinate $z$. 
Using $z$, we define
$$
A_W = 
H^0(C_s\setminus \pi^{-1}(p),\mathcal{O}_{C_s})
\subset \mathbb{C}((z))^{\dsum n},
$$
which is a point of the Grassmannian satisfying
$$
\begin{cases}
\Ker(\rgam_{A_W}) \isom H^0(C_s,\mathcal{O}_{C_s})\\
\Coker(\rgam_{A_W}) \isom H^1(C_s,\mathcal{O}_{C_s}).
\end{cases}
$$
We note that $A_W$ is a ring and $W$ is an $A_W$-module. 
Since 
$$
\Coker(\rgam_{A_W}) = \frac{\mathbb{C}((z))^{\dsum n}}
{A_W + \mathbb{C}[[z]]^{\dsum n}},
$$
$H^1(C_s,\mathcal{O}_{C_s})$ is injectively mapped to 
$\Hom(W,\mathbb{C}((z))^{\dsum n}/W)$.

The  local coordinate $z$ determines 
a local trivialization of $K_C$
$$
 K_C |_{\hat{C}_p} \overset{\sim}{\longrightarrow} \mathcal{O}_{\hat{C}_p}\cdot dz
 \isom \mathcal{O}_{\hat{C}_p} ,
$$
and hence that of $K_{C_s} ^{-1}$
$$
K_{C_s} ^{-1}\big|_{\hat{C}_{s,\pi^{-1}(p)}}
\overset{\sim}{\longrightarrow} \left(\mathcal{O}_{\hat{C}_p}\right)
 ^{\dsum n}
\cdot \frac{\partial}{\partial z} \isom
\mathbb{C}[[z]]^{\dsum n}
\cdot \frac{\partial}{\partial z}.
$$
This trivialization gives
$$
H^1(C_s,K_{C_s} ^{-1}) \isom
\frac{\mathbb{C}((z)) ^{\dsum n} \cdot \partial/\partial z}
{D_W +\mathbb{C}[[z]] ^{\dsum n} \cdot \partial/\partial z},
$$
where 
$$
D_W = H^0(C_s\setminus \pi^{-1}(p), K_{C_s} ^{-1})
\subset \mathbb{C}((z)) ^{\dsum n} \cdot \partial/\partial z.
$$
Note that $s\in V^* _{GL}$ determines an 
element of $H^0(C_s,K_{C_s})$ because
$$
s\in V_{GL} ^* = \bigoplus_{i=1} ^n H^0(C,K_C ^{\tensor i})
\isom H^0(C,\pi_* \pi^* K_C ^{\tensor n})
\isom H^0(C_s,K_{C_s}).
$$
This holomorphic $1$-form on $C_s$ induces a homomorphism
$$\mathcal{L}\rightarrow \mathcal{L}\tensor K_{C_s},
$$
or equivalently, 
$$
K_{C_s} ^{-1} \tensor \mathcal{L} \longrightarrow \mathcal{L}.
$$
In terms of the local coordinate $z$, this homomorphism gives an 
action of $D_W$ on $W$.
Since $D_W\cdot W\subset W$, 
we conclude that $H^1(C_s,K_{C_s} ^{-1}) $ is injectively 
mapped to $T_W Gr_n(\mathbb{C})$. In this construction
$H^1(C_s,\mathcal{O}_{C_s})$ and $H^1(C_s,K_{C_s} ^{-1})$
have no common element in $T_W  Gr_n(\mathbb{C})$
except for $0$.
This establishes the injectivity of $d\mu$.

In \cite{LM2}, it is proved that the orbit of the $n$-component 
KP flows starting from $\overline{W}$ is isomorphic
to $\Pic^{d+(n^2-n)(g-1)}(C_s)$ (Theorem~5.8, \cite{LM2}),
which is the fiber of the Hitchin map $PH$. This completes the
proof of the theorem.
\end{proof}

The above theorem does not say anything about the
Hitchin integrable system, because the partial projective
moduli space is not a symplectic manifold and we do 
not have any integrable systems on it. To directly compare 
the Jacobian flows of the Hitchin systems and the KP flows,
we use the \emph{relative} Grassmannian of \cite{Q}
defined on the scheme $U_p$ of (\ref{eq:unramified}).
So let $\pi^{-1}(p) = \{p_1,\dots,p_n\}$, and denote by $z_i = \pi^*(z)$
the formal coordinate of $\hat{C}_{s,p_i}$. 
Choose a linear coordinate system $(h_1,\dots, h_N)$ 
for $V^* _{GL}$ as in (\ref{eq:hjs}). 
Recall that 
$\pi_*\mathcal{O}_{C_s}\isom
\oplus_{i=0} ^{n-1} K_C ^{\tensor -i}$, and from  (\ref{eq:H1OCs})
we have
$$
H^1(C_s,\mathcal{O}_{C_s})\isom
\bigoplus_{i=0} ^{n-1} H^1(C,K_C ^{\tensor -i}) = V_{GL}.
$$
Using the $\check{\text{C}}$ech cohomology computation 
based on the covering $C = (C\setminus \{p\})\cup \hat{C}_p$,
we expand each $h_k\in V_{GL}$ as
\begin{equation}
\label{eq:hkinz}
h_k = \sum_{i=1} ^n \sum_{j\ge 1}
t_{ijk} z_i ^{-j}, \qquad t_{ijk}\in \mathbb{C}.
\end{equation}
The $n$-component KP flows
on the 
Grassmannian $Gr_n (U_p)$  is defined 
by a formal action of 
\begin{equation}
\label{eq:nKP}
\exp\left(\sum_{i=1} ^n \sum_{j\ge 1}\sum_{k=1} ^N
t_{ijk} z_i ^{-j}\right).
\end{equation}
For degree $d=n(g-1)$, we have the following:

\begin{thm}
\label{thm:embedding1}
Let us choose a point $p\in C$, a formal coordinate
$z$ of $\hat{C}_p$, and a square root of the canonical sheaf
$K_C ^{1/2}$. Then the fibration 
$$
\xymatrix{
		\mathfrak{C}_{V^* _{GL}}(n,n(g-1))\;\ar[dr]_{F}
		 \ar[rr]^{\quad\pi}
		&& 	\;\; C\times V^* _{GL} \ar[dl]^{p_2}\\
		& V^* _{GL} &
		}
$$
determines a point $W\in Gr_n ^0(U_p)$.
There is a birational map  from $\mathcal{H}_C(n,n(g-1))$ to
the orbit of the $n$-component KP flows on the quotient
Grassmannian $Z_n ^0(U_p)$ starting from $\overline{W}$.
The Hitchin integrable system on the Higgs moduli space
$\mathcal{H}_C(n,n(g-1))$ is the pull-back of the $n$-component
KP equations via this birational map.
\end{thm}

\begin{proof}
We regard $C, p, z$, and $K_C ^{1/2}$ being defined over the trivial family
$C\times V^* _{GL}$.  There is a natural choice of a 
local trivialization of $K_C ^{1/2}$ on $\hat{C}_p$ determined by the
coordinate $z$. Note that  
\begin{equation}
\label{eq:sqrtcanonicalonspec}
\mathcal{K}^{1/2} \overset{\text{def}}{=} \pi^*(K_C ^{n/2})
\end{equation}
is a square root of the relative dualizing sheaf $\mathcal{K}
= \pi^*(K_C ^{n})$ on 
$$\mathfrak{C}_{V^* _{GL}}(n,n(g-1)).
$$
Take an arbitrary point $s\in U_p$.  Since $p\in C$ is a point
of $C$ at which $\pi$ is unramified, the formal coordinate $z$ determines  a local trivialization of $\mathcal{K} ^{1/2}$
along $\pi^{-1}(p)\subset \mathfrak{C}_{V^* _{GL}}(n,n(g-1))$.
Define
\begin{equation}
\label{eq:Ws}
W = H^0(\mathfrak{C}_{V^* _{GL}}(n,n(g-1))\setminus \pi^{-1}(p), \mathcal{K}^{1/2}) \subset
\mathcal{O}_{U_p}((z)) ^{\dsum n},
\end{equation}
using this local trivialization.
Since
$$
\begin{cases}
\Ker(\rgam_{W_s}) \isom H^0(C_s,K_{C_s} ^{1/2})\\
\Coker(\rgam_{W_s}) \isom H^1(C_s,K_{C_s} ^{1/2})
\end{cases}
$$
for each $s\in U_p$, 
we have $W\in Gr_n ^0(U_p)$. 
Let us analyze the action of (\ref{eq:nKP}) at $W$. 
On each $C_s$, $h_k = \sum t_{ijk}  z_i ^{-j}$ defines an
element of 
$H^1(C_s,\mathcal{O}_{C_s})$.   The exponential of
(\ref{eq:nKP})
is the Lie integration map
\begin{equation}
\label{eq:exptojacCs}
V_{GL}\isom H^1(C_s,\mathcal{O}_{C_s})
\overset{\exp}{\longrightarrow}
H^1(C_s,\mathcal{O}_{C_s})\big/H^1(C_s,\mathbb{Z})
=\Jac(C_s).
\end{equation}
Therefore, the action of (\ref{eq:nKP}) at the point  $W$ produces
simultaneous deformations of the line bundle
\begin{equation}
\label{eq:nKPevolution}
\mathcal{K}^{1/2}\longmapsto
\mathcal{K}^{1/2}\tensor \exp\left(\sum_{i=1} ^n \sum_{j\ge 1}\sum_{k=1} ^N
t_{ijk} z_i ^{-j}\right)
\end{equation}
on $\mathfrak{C}_{V^* _{GL}}(n,n(g-1))$. 
Consequently, the orbit of the $n$-component KP flows
on $Z_n ^0(U_p)$ starting at $\overline{W}$ is 
naturally identified with the family
of $\Pic^{n^2(g-1)}(C_s)$ on $U_p$. 
Since the Hitchin fibration
$H: \mathcal{H}_C(n,n(g-1))\rightarrow V^* _{GL}$ is
also a family of $\Pic^{n^2(g-1)}(C_s)$ on $V^* _{GL}$,
we have a rational map from $H: \mathcal{H}_C(n,n(g-1))$ to 
the $n$-component KP orbit of $\overline{W}$
in $Z_n ^0(U_p)$. 

Recall that the Hitchin integrable system on 
$H: \mathcal{H}_C(n,n(g-1))$ is the linear Jacobian
flows defined by the coordinate functions $(h_1,\dots,h_N)$.
We note that this is exactly what (\ref{eq:nKPevolution})
gives. 
\end{proof}

\begin{rem}
So far we have assumed  that
$C_s$ is non-singular. Although it will not be an Abelian
variety, we can still define the Jacobian
$\Jac(C_s)$ using (\ref{eq:exptojacCs}) when the spectral 
curve $C_s$ is singular. As long as we choose $p\in C$ so that
$\pi^{-1}(p)$ avoids the singular locus of $C_s$, 
the KP flows are well defined. Moreover, the theory of
\emph{Heisenberg
KP flows} introduced in \cite{AdBe, LM2} allows us to consider
the covering $\pi:C_s\rightarrow C$ right at a ramification point. 
It is more desirable to define the Grassmmanian over the whole
$V^* _{GL}$ and to deal with the entire family 
$F:\mathfrak{C}_{V^* _{GL}}\rightarrow V^* _{GL}
$ together with the moduli stack of
Higgs bundles instead of the stable moduli we have considered here,
since the framework of the Sato Grassmannian allows us to 
consider all vector bundles on $C$ and degenerated spectral curves.
We refer to \cite{DM} for a study in this direction.
\end{rem}

\section{Serre duality and formal adjoint of pseudo-differential
operators}
\label{sect:Serre}

In this section we analyze the Serre duality of Higgs bundles
in terms of the language of the Sato Grassmannian, and
identify the involution on the corresponding pseudo-differential
operators.

The Krichever construction (\cite{M1984, M1990, SW}) assigns a 
point $W$ of the Sato Grassmannian,
\begin{equation}
\label{eq:Krichever1}
Gr_n(\mathbb{C})\owns W=\beta(H^0(C\setminus\{p\},F))\subset
\mathbb{C}((z))^{\dsum n},
\end{equation}
to
a set of geometric data $(C,p,z,F,\beta)$, where $C$ is an 
irreducible algebraic curve, $p\in C$ is a non-singular point,
$z$ is a formal parameter of the completion $\hat{C}_p$,
$F$ is a torsion-free sheaf of rank $n$ on $C$, and 
$$
\beta: F|_{\hat{C}_p} \overset{\sim}{\longrightarrow} \mathcal{O}_{\hat{C}_p} ^{\dsum n}
= \mathbb{C}[[z]]^{\dsum n}
$$
is a local trivialization of $F$ around $p$.  We continue to 
assume that $C$ is non-singular, hence $F$ is locally-free on
$C$. 
As noted in Section~\ref{sect:KP}, 
the choice of a local coordinate $z$ automatically determines 
a local trivialization of $K_C$:
\begin{equation}
\label{eq:KCtrivial}
\alpha: K_C |_{\hat{C}_p} \overset{\sim}{\longrightarrow} \mathcal{O}_{\hat{C}_p}\cdot dz
 \isom \mathcal{O}_{\hat{C}_p}.
 \end{equation}
For every homomorphism $\xi: F\rightarrow K_C$, we 
have an element  
\begin{equation*}
\beta^*(\xi) \overset{\text{def}}{=}\alpha\circ \xi\circ \beta^{-1}\in \left(\mathcal{O}_{\hat{C}_p} ^{\dsum n}\right)^*
= \mathcal{O}_{\hat{C}_p} ^{\dsum n}
\end{equation*}
that is determined by the commutative diagram
\begin{equation}
\label{eq:beta*}
\begin{CD}
F|_{\hat{C}_p} @>{\xi}>> K_C|_{\hat{C}_p}\\
@V{\beta}V{\wr}V @V{\wr}V{\alpha}V\\
\mathcal{O}_{\hat{C}_p} ^{\dsum n} @>{\beta^*(\xi)}>> \mathcal{O}_{\hat{C}_p}.
\end{CD}
\end{equation} 
Thus we have the canonical choice of the local trivialization
$$
\beta^*: F^*\tensor K_C|_{\hat{C}_p}
\overset{\sim}{\longrightarrow}\mathcal{O}_{\hat{C}_p}
^{\dsum n}.
$$

\begin{Def}[Serre Dual] The \emph{Serre dual} of the 
geometric date $(C,p,z,F,\beta)$ is the set of data
$(C,p,z,F^*\tensor K_C, \beta^*)$.
\end{Def}

To identify the counterpart of the Serre duality on the Sato
Grassmannian, we introduce a non-degenerate symmetric paring in
$\mathbb{C}((z))$ by
\begin{equation}
\label{eq:pairing}
\langle a(z), b(z) \rangle = \frac{1}{2\pi i}
\oint a(z) b(z) dz
= \text{ the coefficient of $z^{-1}$ in } a(z) b(z),
\end{equation}
and define
\begin{equation}
\label{eq:pairingn}
\langle f(z), g(z) \rangle_n  = \sum_{i=1} ^n \langle f_i(z),
g_i(z)\rangle ,
\qquad f(z), g(z)\in \mathbb{C}((z))^{\dsum n}.
\end{equation}

\begin{lem}
\label{lem:WandWbot}
Let $W\in Gr_n(\mathbb{C})$ be a point of the Sato Grassmannian
of index $d$. Then 
$$
W^\bot = \{f\in \mathbb{C}((z))^{\dsum n} \;|\; \langle f,w \rangle_n = 0
\text{ for all } w\in W\}
$$
is a point of $Gr_n(\mathbb{C})$ of index $-d$. Moreover, 
we have
\begin{equation}
\label{eq:Ker=Coker*}
\begin{cases}
\Ker(\rgam_W)^* \isom \Coker(\rgam_{W^\bot})\\
\Coker(\rgam_W)^* \isom \Ker(\rgam_{W^\bot}),
\end{cases}
\end{equation}
where $\rgam_W$ is defined by the exact sequence
$$
0\longrightarrow \Ker(\rgam_W) \longrightarrow
W 
\overset{\rgam_W}{\longrightarrow} \mathbb{C}((z))^{\dsum n}/\mathbb{C}[[z]]^{\dsum n}
\longrightarrow \Coker(\rgam_W) \longrightarrow 0.
$$
\end{lem}

\begin{proof}
Take an element $f\in W^\bot \cap \mathbb{C}[[z]]^{\dsum n}$. Then
$$\langle f, W+\mathbb{C}[[z]]^{\dsum n}\rangle_n = 0. 
$$
Thus it induces a
linear map
$$
f: \mathbb{C}((z))\big/ (W+ \mathbb{C}[[z]]^{\dsum n})
\owns g
\longmapsto \langle f,g \rangle_n\in 
\mathbb{C},
$$
defining a natural inclusion 
\begin{equation}
\label{eq:KerinCoker}
\Ker(\rgam_{W^\bot})\subset
\Coker(\rgam_W)^*. 
\end{equation}
Now let $h^*\in \Coker(\rgam_W)^*$
be an arbitrary element.
Choose a sequence 
$$\{g_1 ^i,g_2 ^i, g_3 ^i\dots\}_{i=1,\dots,n}$$
of elements of $\mathbb{C}((z))^{\dsum n}$ such that
\begin{enumerate}
\item $g_j ^i = \mathbf{e}^i z^{-j} + \sum_{\ell = 1} ^n \sum_{k\ge 1}
 {g}_{\ell jk} \mathbf{e}^\ell z^{-j+k}$;
\item if $W$ has an element whose leading term is
$\mathbf{e}^i z^{-j}$ for $j>0$,
then $g_j ^i\in W$.
\end{enumerate}
Here $\mathbf{e}^i = (0,\dots,0,1,0,\dots,0)^T$ is the 
$i$-th standard basis vector for $\mathbb{C}^n$. 
We denote by  $\bar{g}_j ^i$ the projection image of $g_j ^i$
in $\mathbb{C}((z))^{\dsum n}\big/ (W+ \mathbb{C}[[z]]^{\dsum n})$.
Consider the set of equations
\begin{equation}
\label{eq:eqforf}
\langle f, g_j ^i\rangle_n = h^*(\bar{g}_j ^i),\qquad i=1,\dots, n; \quad j=1,2,3,\dots
\end{equation}
for $f\in \mathbb{C}[[z]]^{\dsum n}$. If we write $f(z) = \sum_{i=1} ^n
\sum_{j\ge 0}
 a_{ij}  \mathbf{e}^i z^j$, then (\ref{eq:eqforf}) is equivalent to 
 $$
 \begin{aligned}
 a_{i0} &= h^*(g_1 ^i)\\
 a_{i1} &= h^*(g_2 ^i) -\sum_{\ell = 1} ^n a_{\ell 0}  g_{\ell 21}\\
 a_{i2} &=h^*(g_3 ^i)-\left(\sum_{\ell = 1} ^n 
 a_{\ell 1} g_{\ell 31} + a_{\ell 0} g_{\ell 32} \right)\\
 a_{i3}  &=h^*(g_4 ^i)-\left(\sum_{\ell = 1} ^n
 a_{\ell 2} g_{\ell 41} + a_{\ell 1} g_{\ell 42} + a_{\ell 0} g_{\ell 43}
 \right)\\
 &\vdots
\end{aligned}
$$
It is obvious that (\ref{eq:eqforf}) has a unique solution.
By construction we have $\langle f, W\rangle_n = 0$, hence
$
f\in \Ker(\rgam_{W^\bot})
$.
Therefore, the inclusion
(\ref{eq:KerinCoker}) is indeed a surjective map.

The application of the same argument to $W\cap
\mathbb{C}[[z]]^{\dsum n}$  establishes that
$$\Ker(\rgam_W)\isom \Coker(\rgam_{W^\bot})^*.
$$
\end{proof}

\begin{rem}
We note that $W^{\bot\bot} = W$. It is obvious that
$W\subset W^{\bot\bot}$. The relation (\ref{eq:Ker=Coker*})
then makes the inclusion relation actually the equality.
\end{rem}

\begin{thm}
\label{thm:Serre=bot}
Let $W\in Gr_n(\mathbb{C})$ be the point of the Grassmannian
corresponding to a set of geometric data
$(C,p,z,F,\beta)$. Then the point corresponding
to its Serre dual $(C,p,z,F^*\tensor K_C,\beta^*)$ is $W^\bot$.
\end{thm}

\begin{proof}
Let $W'$ be the point of the Grassmannian corresponding 
to $(C,p,z,F^*\tensor K_C,\beta^*)$. 
Note that 
$$
\index(\rgam_W) = \dim H^0(C,F)-\dim H^1(C,F)
= -\index(\rgam_{W'}).
$$
Take an arbitrary $f\in H^0(C\setminus\{p\},F)$ and
$g \in  H^0(C\setminus \{p\}, F^*\tensor K_C)$. Then
$$
g\cdot f \in H^0(C\setminus\{p\},K_C).
$$
From (\ref{eq:beta*}) we see that
$\beta^*(g)\cdot \beta(f) = \alpha(g\cdot f)$.
Abel's theorem tells us that the sum of the residues of the meromorphic
$1$-form 
$g\cdot f$ on $C$ is $0$. Since $g\cdot f$ is holomorphic everywhere
on $C$ except for $p$, we have
$$
0 = \res_p(g\cdot f) = \langle \beta^*(g), \beta(f)\rangle_n.
$$
Therefore, $W'\subset W^\bot$. Since $\rgam_{W'}$ and 
$\rgam_{W^\bot}$ have the isomorphic
kernels and cokernels, we conclude that $W'=W^\bot$. 
\end{proof}

The ring of \emph{ordinary differential operators} is defined 
as
\begin{equation}
\label{eq:differential}
\mathcal{D}= \mathbb{C}[[x]] \left[\frac{d}{dx}\right].
\end{equation}
Extending the powers of the differentiation to negative integers,
we define the ring of \emph{pseudo-differential operators}
by
\begin{equation}
\label{eq:pseudo}
\mathcal{E} = \mathcal{D}+ \mathbb{C}[[x]] \left[\left[
\left(\frac{d}{dx}\right)^{-1}\right]\right].
\end{equation}
Let $\mathcal{E}x$ be the left maximal ideal of 
$\mathcal{E}$ generated by $x$. Then $\mathcal{E}/\mathcal{E}x$
is a left $\mathcal{E}$-module. Let us denote
$\partial = d/dx$. As a $\mathbb{C}$-vector space we identify
\begin{equation}
\label{eq:EEx}
\mathcal{E}/\mathcal{E}x = \mathbb{C} ((\partial^{-1})).
\end{equation}
Two useful formulas for calculating
pseudo-differential operators are
\begin{equation}
\label{eq:formula1}
\partial^n\cdot a(x) = \sum_{i\ge 0} \binom{n}{i}
a^{(i)}(x)\partial^{n-i},
\end{equation}
where $a^{(i)}(x)$ is the $i$-th derivative of $a(x)$, and
\begin{equation}
\label{eq:formula2}
a(x)\cdot \partial^n = \sum_{i\ge 0} (-1)^i \binom{n}{i}
\partial^{n-i}\cdot a^{(i)}(x).
\end{equation}
A pseudo-differential operator has thus two expressions
\begin{align*}
P = \sum_{n\in \mathbb{Z}} a_n(x) \partial^n
&= \sum_{n\in \mathbb{Z}} \sum_{i\ge 0} (-1)^i\binom{n}{i} \partial^{n-i}\cdot a_n ^{(i)}
(x)\\
&= \sum_{m\in \mathbb{Z}} \sum_{i\ge 0} (-1)^i \binom{m+i}{i}\partial^m\cdot
a_{m+i} ^{(i)}(x).
\end{align*}
The natural projection $\mathcal{E}\rightarrow \mathcal{E}/\mathcal{E}x$ is given by
\begin{multline}
\label{eq:rho}
\rho:\mathcal{E}\owns P = \sum_n a_n(x) \partial^n\\
\longmapsto \sum_m \sum_{i\ge 0} (-1)^i\binom{m+i}{i} 
a_{m+i} ^{(i)}(0)\partial^{m} \in 
\mathbb{C} ((\partial^{-1})).
\end{multline}
The relation to the Grassmannian comes from the identification
\begin{equation}
\label{eq:identification}
\mathbb{C}((z))\owns f(z)\longmapsto
f(\partial^{-1})\cdot \partial^{-1}\in
\mathbb{C} ((\partial^{-1})).
\end{equation}
Then $\mathbb{C}((z))$ becomes an $\mathcal{E}$-module. 
The action of $P\in\mathcal{E}$ on $f(z)\in \mathbb{C}((z))$
is given by the formula
$$
P\cdot f(z) = \rho(P \cdot f(\partial^{-1})\partial^{-1})
\in \mathbb{C}((z)).
$$

\begin{Def}[Adjoint] 
The \emph{adjoint} of $P=\sum_\ell a_\ell (x)\partial ^\ell$
is defined by
$$
P^* = \sum_\ell \partial^\ell\cdot  a_\ell (-x).
$$
\end{Def}

\begin{rem}
Our adjoint is slightly different from the 
 \emph{formal adjoint} of \cite{JM, T}, which  is defined to be
$\sum_\ell (-\partial)^\ell\cdot  a_\ell (x)$. This is due to 
the fact that we are considering the action of pseudo-differential
operators on the function space $\mathbb{C}((z))$ through
 \emph{Fourier transform}. Thus $\partial$
 acts as the multiplication of $z^{-1}$, 
 and $x$ acts as the differentiation.
\end{rem}

Let us compute the $(m,n)$-matrix entry of $P$. We find
\begin{equation*}
\begin{aligned}
&\langle z^m, P\cdot z^n\rangle\\
&= \left\langle z^m , \rho\left(\sum_\ell a_\ell(x) \partial^{\ell -n-1}\right)\right\rangle\\
&=\left\langle z^m , \rho\left(\sum_\ell \sum_i (-1)^i\binom{\ell-n-1}{i}
\partial^{\ell -n-1-i}\cdot a_\ell ^{(i)}(x)\right)\right\rangle\\
&=\left\langle z^m , \sum_\ell \sum_i (-1)^i\binom{\ell-n-1}{i}
 a_\ell ^{(i)}(0)\cdot z^{-\ell +n+i}\right\rangle\\
 &= \sum_i (-1)^i \binom{m+i}{i}a_{m+n+i+1} ^{(i)}(0).
\end{aligned}
\end{equation*}
Similarly, we find
\begin{equation*}
\begin{aligned}
&\langle z^n, P^*\cdot z^m\rangle\\
&= \left\langle z^n , \rho\left(\sum_\ell \partial^\ell a_\ell(-x) 
\cdot \partial^{-m-1}\right)\right\rangle\\
&=\left\langle z^n , \rho\left(\sum_\ell \sum_i (-1)^{i}(-1)^i
\binom{-m-1}{i}
\partial^{\ell -m-1-i}\cdot a_\ell ^{(i)}(-x)\right)\right\rangle\\
&=\left\langle z^n , \sum_\ell \sum_i \binom{-m-1}{i}
 a_\ell ^{(i)}(0)\cdot z^{-\ell +m+i}\right\rangle\\
 &= \sum_i  \binom{-m-1}{i}a_{m+n+i+1} ^{(i)}(0).
\end{aligned}
\end{equation*}

Following \cite{LM2}, we introduce the ring of 
differential and pseudo-differential
operators with matrix coefficients and denote them
by $gl_n(\mathcal{D})$ and
$gl_n(\mathcal{E})$, respectively. The adjoint of a matrix pseudo-differential
operator is defined in an obvious way:
\begin{equation}
\label{eq:matrixadjoint}
P = [P_{ij}]_{i,j=1,\dots,n} \Longrightarrow
P^* \overset{\text{def}}{=} [P^* _{ji}]_{i,j=1,\dots,n}.
\end{equation}

\begin{prop}
\label{prop:adjointformula}
For arbitrary $f(z), g(z)\in \mathbb{C}((z))^{\dsum n}$ and $P\in gl_n(\mathcal{E})$
we have the adjoint formula
$$
\langle f(z), P\cdot g(z)\rangle_n = \langle P^*\cdot f(z), g(z)\rangle_n.
$$
\end{prop}

\begin{proof}
The statement follows from the above computations, noting the identity
$$
(-1)^i \binom{m+i}{i} = \binom{-m-1}{i},
$$
and  the usual matrix transposition. 
\end{proof}

A more direct relation between pseudo-differential operators
and the Sato Grassmannian is that the identification of the
\emph{big cell} of $Gr_n(\mathbb{C})$ with the group 
\begin{equation}
\label{eq:groupofpdo}
G = I_n + gl_n(\mathbb{C}[[x]]) [[\partial^{-1}]]\cdot \partial^{-1}
\end{equation}
of monic pseudo-differential operators of order $0$ (\cite{
LM2}, Theorem~6.5).

\begin{thm}
\label{thm:bigcell}
Let $W\in Gr_n(\mathbb{C})$ be in the big cell, and correspond 
to a pseudo-differential operator $S\in G$. Then $W^\bot$ 
corresponds to $(S^* )^{-1}\in G$. 
\end{thm}

\begin{proof}
The correspondence of $S\in G$ and $W$ in the big cell is
the equation
$$
W = S^{-1}\cdot \mathbb{C}[z^{-1}]^{\dsum n}\cdot z^{-1}.
$$
For $f(z),g(z)\in \mathbb{C}[z^{-1}]^{\dsum n}\cdot z^{-1}$, we have
$$
\langle S^*\cdot f(z),S^{-1}\cdot g(z)\rangle_n
= \langle f(z),SS^{-1}\cdot g(z)\rangle_n = 0.
$$
\end{proof}

The \emph{Lax operator} is defined by
\begin{equation}
\label{eq:Laxoperator}
\mathbf{L} = S\cdot I_n\cdot \partial\cdot S^{-1}.
\end{equation}
Lax operators bijectively correspond to points of the big cell
of the quotient Grassmannian $Z_n(\mathbb{C})$. 
If $\overline{W}$ corresponds to $\mathbf{L}$, then
$\overline{W}^\bot$ corresponds to $\mathbf{L}^*$.
Therefore, the Serre duality indeed corresponds to the
adjoint action of the pseudo-differential operators.

\section{The Hitchin integrable system for $Sp_{2m}$ and
corresponding KP-type equations}
\label{sect:Sp}

In this section we identify the KP-type equations that are equivalent to the
Hitchin integrable system for $Sp_{2m}$. Since the 
moduli space of the $Sp_{2m}$-Higgs bundles is identified
as the fixed-point-set of the Serre duality involution on
$\mathcal{H}_C(2m,2m(g-1))$ (Proposition~\ref{prop:SerreandSp}), 
we will see that the evolution equations are the KP equations
restricted to a certain type of self-adjoint Lax operators.

Let $(E,\phi) = (E^*\tensor K_C, -\phi^*)$ be a fixed point
of the Serre duality operation on the Higgs moduli space
$\mathcal{H}_C(2m,2m(g-1))$, and $(\mathcal{L}_E,C_s)$ the
corresponding spectral data. We choose a point $p\in C$ and 
a local coordinate $z$ as in Section~\ref{sect:KP}. We assume
that the spectral cover $\pi:C_s\rightarrow C$ is unramified at $p$.
This time we choose a local trivialization
$$
\beta: \mathcal{L}_E |_{\hat{C}_{s,\pi^{-1}(p)}}
\overset{\sim}{\longrightarrow} \mathcal{O}_{\hat{C}_{s,\pi^{-1}(p)}}
\isom \mathbb{C}[[z]]^{\dsum 2m}.
$$
Then the Serre dual of the data $(\pi_s:C_s \rightarrow C,p,z,
\mathcal{L}_E,\phi,\beta)$ is 
$$
(\pi_{s^*}:C_{s^*} \rightarrow C,p,z,\epsilon^*(\mathcal{L}_E ^*\tensor K_{C_s}), -\phi^*, \epsilon^*\beta^*).
$$
In general the dual trivialization is defined by
$$
\begin{CD}
\mathcal{L}_E |_{\hat{C}_{s,\pi_s ^{-1}(p)}}
@>{\xi}>> K_{C_s} |_{\hat{C}_{s,\pi_s ^{-1}(p)}}\\
@V{\beta}V{\wr}V @V{\wr}V{\alpha}V\\
\mathcal{O}_{\hat{C}_{s,\pi_s ^{-1}(p)}}
@>{\beta^*(\xi)}>> 
\mathcal{O}_{\hat{C}_{s,\pi_s ^{-1}(p)}}\\
@V{\epsilon}V{\wr}V @V{\wr}V{\epsilon}V\\
\mathcal{O}_{\hat{C}_{s^*,\pi_{s^*} ^{-1}(p)}}
@>{\epsilon^*\beta^*(\xi)}>> 
\mathcal{O}_{\hat{C}_{s^*,\pi_{s^*} ^{-1}(p)}}.
\end{CD}
$$
Note that in the current case we have $s=s^*$ and
$\epsilon : C_s\rightarrow C_s$ is a non-trivial involution. 
Since $\epsilon$ commutes with $\pi$, it induces
an involution of the set $\pi^{-1}(p) = \{p_1,\dots,p_{2m}\}$.
Let us number the $2m$ distinct points so that 
$$
\epsilon (p_1,\dots,p_{2m}) = (p_1,\dots,p_{2m})\cdot A,
$$
where 
\begin{equation}
\label{eq:matrixA}
A = \begin{bmatrix}
0&I_m\\
I_m &0
\end{bmatrix}.
\end{equation}
By the same argument we used in Theorem~\ref{thm:Serre=bot},
if we
define 
$$
W = \beta(H^0(C_s\setminus \pi^{-1}(p),\mathcal{L}_E))
\subset \mathbb{C}((z))^{\dsum 2m},
$$
then we have
\begin{equation}
\label{eq:WbotinSp}
AW^\bot = \epsilon^*\beta^*(H^0(C_{s^*}\setminus \pi^{-1}(p),
\epsilon^*(\mathcal{L}^* _E\tensor K_{C_s})))
\subset \mathbb{C}((z))^{\dsum 2m}.
\end{equation}

Let us assume that $E$ is not on the theta divisor of 
\cite{BNR} so that $H^0(C,E) = H^1(C,E)=0$. 
Then $W\in Gr_{2m}(\mathbb{C})$ is on the big cell,
and hence corresponds to a monic $0$-th order pseudo-differential
operator $S\in G$.  Since $W=AW^\bot$, we have 
\begin{equation}
\label{eq:twistedselfadjoint}
S=A\cdot (S^*)^{-1}
\cdot A. 
\end{equation}
The Lax operator 
$$
\mathbf{L} = S\cdot I_{2m}\cdot \partial \cdot S^{-1}
=
\begin{bmatrix}
\mathbf{L}_1&\mathbf{L}_2\\
\mathbf{L}_3&\mathbf{L}_4
\end{bmatrix}
$$
then satisfies that 
\begin{equation}
\label{eq:Laxadjoint}
\mathbf{L}^* = A\cdot \mathbf{L}\cdot A,
\end{equation}
or equivalently
$$
\mathbf{L}_1 ^* = \mathbf{L}_4, 
\qquad 
\mathbf{L}_2 ^* = \mathbf{L}_2, 
\qquad
\mathbf{L}_3 ^* = \mathbf{L}_3.
$$
The time evolution of $S$ and the Lax operator $\mathbf{L}$ is 
given by the formula established in Section~6 of \cite{LM2}.
Let $S(t)$ be the solution of the $n$-component KP equations
with the initial data $S(0) = S$. Then $S(t)$ is given by the
generalized Birkhoff decomposition of \cite{LM2, M1988}
\begin{equation}
\label{eq:birkhoff}
\exp\left(\begin{bmatrix}
D_1\\
&D_2
\end{bmatrix}\right)\cdot S(0)^{-1} = 
S(t)^{-1}\cdot Y(t),
\end{equation}
where $D_1$ and $D_2$ are diagonal matrices of the shape
$$
D_i = \begin{bmatrix}
\sum_{j\ge 1}t_{ij1} \partial^j\\
&\ddots\\
&&\sum_{j\ge 1}t_{ijm} \partial^j
\end{bmatrix}
$$
corresponding to (\ref{eq:nKP}), and $Y(t)$ is an invertible 
infinite order
differential operator introduced in \cite{M1988}. 
We impose  that the time evolution $S(t)$ satisfies 
the same twisted self-adjoint condition (\ref{eq:twistedselfadjoint}). 
Applying the adjoint-inverse operation and conjugation by 
$A$ of (\ref{eq:matrixA}) to (\ref{eq:birkhoff}), we obtain
$$
\exp\left(A\begin{bmatrix}
-D_1\\
&-D_2
\end{bmatrix}A\right)\cdot S(0)^{-1} = 
S(t)^{-1}\cdot (A (Y(t)^*)^{-1}A).
$$
Therefore, the time evolution (\ref{eq:birkhoff}) 
with the condition $D_2=-D_1$ preserves the twisted self-adjointness
(\ref{eq:twistedselfadjoint}). We have thus established

\begin{thm}
\label{thm:SpHitchinandKP}
The KP-type equations that generate the Hitchin integrable
systems on the moduli spaces of $Sp_{2m}$-Higgs bundles
are the reduction of the $2m$-component KP equations 
that preserve the twisted self-adjointness $(\ref{eq:twistedselfadjoint})$
for the operator $S$, or $(\ref{eq:Laxadjoint})$ for the 
Lax operator $\mathbf{L}$. The time evolution $S\longmapsto S(t)$
preserving the condition is given by
the generalized Birkhoff decomposition
\begin{equation}
\label{eq:SpKP}
\exp\left(\begin{bmatrix}
D\\
&-D
\end{bmatrix}\right)\cdot S(0)^{-1} = 
S(t)^{-1}\cdot Y(t),
\end{equation}
where 
$$
D = \begin{bmatrix}
\sum_{j\ge 1}t_{j1} \partial^j\\
&\ddots\\
&&\sum_{j\ge 1}t_{jm} \partial^j
\end{bmatrix}.
$$
\end{thm}

\begin{rem}
Since the time evolution of (\ref{eq:SpKP}) is given by 
a traceless matrix $\text{diag}(D,-D)$, every finite-dimensional
orbit of this system is a Prym variety by the general theory
of \cite{LM2}, which is expected from the $Sp$ Hitchin fibration.
\end{rem}

\bigskip


\end{document}